\documentclass[11pt]{amsart}
\usepackage[
  left=1.15in,
  right=1.15in,
  top=1.0in,
  bottom=1.05in
]{geometry}
\usepackage{amsmath,amsthm,amssymb,mathtools,amsfonts,amsopn,amscd}
\usepackage{bm}
\usepackage{tikz}
\usepackage{tikz-cd}    
\usepackage{rotating}
\usepackage{graphicx}
\usepackage{etoolbox}
\usepackage{enumitem}
\usepackage{longtable}
\usepackage{booktabs} 
\usepackage{tabularx} 
\usepackage{comment}
\usepackage{hyperref}
\usepackage{subfiles}
\usepackage{calc}
\usepackage[numbered]{bookmark}  
\usepackage{mathrsfs}
\usepackage[sc]{mathpazo}
\usepackage{euscript}

\usepackage[noabbrev]{cleveref}
\crefname{section}{Section}{Sections}
\Crefname{section}{Section}{Sections}
\crefname{subsection}{Section}{Sections}
\Crefname{subsection}{Section}{Sections}
\crefname{equation}{Equation}{Equations}
\Crefname{equation}{Equation}{Equations}
\crefname{figure}{Figure}{Figures}
\Crefname{figure}{Figure}{Figures}
\crefname{table}{Table}{Tables}
\Crefname{table}{Table}{Tables}
\crefname{thm}{Theorem}{Theorems}
\Crefname{thm}{Theorem}{Theorems}
\crefname{lem}{Lemma}{Lemmas}
\Crefname{lem}{Lemma}{Lemmas}
\crefname{prop}{Proposition}{Propositions}
\Crefname{prop}{Proposition}{Propositions}
\crefname{cor}{Corollary}{Corollaries}
\Crefname{cor}{Corollary}{Corollaries}
\crefname{df}{Definition}{Definitions}
\Crefname{df}{Definition}{Definitions}
\crefname{ex}{Example}{Examples}
\Crefname{ex}{Example}{Examples}
\crefname{rmk}{Remark}{Remarks}
\Crefname{rmk}{Remark}{Remarks}
\crefname{conj}{Conjecture}{Conjectures}
\Crefname{conj}{Conjecture}{Conjectures}
\crefname{clm}{Claim}{Claims}
\Crefname{clm}{Claim}{Claims}

\setlist[enumerate,1]{label=(\roman*)}

\AddToHook{env/lem/begin}{\crefalias{thm}{lem}}
\AddToHook{env/prop/begin}{\crefalias{thm}{prop}}
\AddToHook{env/cor/begin}{\crefalias{thm}{cor}}
\AddToHook{env/df/begin}{\crefalias{thm}{df}}
\AddToHook{env/ex/begin}{\crefalias{thm}{ex}}
\AddToHook{env/rmk/begin}{\crefalias{thm}{rmk}}
\AddToHook{env/clm/begin}{\crefalias{thm}{clm}}
\AddToHook{env/conj/begin}{\crefalias{thm}{conj}}

\definecolor{secondaryColor}{RGB}{0, 0, 170}
\hypersetup{
	colorlinks=true,
	linkcolor=blue,
	urlcolor=secondaryColor,
	citecolor=secondaryColor,
	linktoc=page,
}

\theoremstyle{plain}
\newtheorem{thm}{Theorem}[section]
\newtheorem{lem}[thm]{Lemma}
\newtheorem{prop}[thm]{Proposition}

\newtheorem{conj}[thm]{Conjecture}

\theoremstyle{definition}
\newtheorem{df}[thm]{Definition}

\theoremstyle{remark}
\newtheorem{rmk}[thm]{Remark}

\newcommand{\ZZ}{\mathbb{Z}}
\newcommand{\NN}{\mathbb{N}}
\newcommand{\QQ}{\mathbb{Q}}
\newcommand{\RR}{\mathbb{R}}
\newcommand{\CC}{\mathbb{C}}

\newcommand{\FF}{\mathbb{F}}

\newcommand{\mfk}{\mathfrak{m}}

\newcommand{\Hfk}{\mathfrak{H}}

\newcommand{\Ical}{\mathcal{I}}

\newcommand{\Rcal}{\mathcal{R}}

\newcommand{\Zcal}{\mathfrak{Z}}

\newcommand{\id}{\operatorname{id}}

\let\oldforall\forall
\renewcommand{\forall}{\oldforall \: }
\let\oldexist\exists
\renewcommand{\exists}{\oldexist \: }

\newcommand{\dep}{\operatorname{dep}}
\newcommand{\wt}{\operatorname{wt}}

\newcommand{\BC}{\operatorname{BC}}

\newcommand{\ww}[1]{\mathfrak{#1}}

\DeclareMathOperator{\Span}{span}

\newcommand{\ov}[1]{\overline{#1}}
\newcommand{\ant}{{\ast-\mathrm{inv}}}
\newcommand{\ev}{\mathrm{ev}}
\newcommand{\sA}{\mathscr{A}}
\newcommand{\Li}{\mathrm{Li}}
\allowdisplaybreaks

\makeatletter

\newcommand{\myToC}{{
		\renewcommand{\contentsname}{}
		\@starttoc{toc}{\contentsname}
}}

\patchcmd{\@tocline}
{\hfil}
{\leaders\hbox{\,.\,}\hfil}

\makeatother

\title[Mishiba's Conjecture on the Coaction of $\infty$-adic Multiple Zeta Values]{Mishiba's Conjecture on the Coaction of $\infty$-adic Multiple Zeta Values}

\author{Hung-Chun Tsui}
\address{(Hung-Chun Tsui) Department of Mathematics, National Tsing Hua University, No. 101, Sec. 2, Guangfu Rd., East Dist., Hsinchu City 300044, Taiwan (R.O.C.)}
\email{hctsui@gapp.nthu.edu.tw}

\date{\today}


\begin{document}
\begin{abstract}
We study multiple zeta values in positive characteristic. We construct, using $\infty$-adic multiple zeta values, a coaction of the $\infty$-adic multiple zeta value algebra modulo $\zeta_A(q-1)$ on the $\infty$-adic multiple zeta value algebra itself, and prove that it agrees with Mishiba's coaction constructed via special values of Carlitz multiple polylogarithms.
We also determine the coaction on the $\ast$-inverse values and show that the antipode on the $\infty$-adic multiple zeta value algebra modulo $\zeta_A(q-1)$ is given by the $\ast$-inverse operation. In particular, these results prove Mishiba's conjecture concerning the coaction and the antipode for $\infty$-adic multiple zeta values.
\end{abstract}

\maketitle
\tableofcontents
\section{Introduction}
\subsection{Classical Background}
\label{sec:classical-background}
Let $\NN$ denote the set of positive integers and $\ZZ$ the set of
integers. We let
\[
\Ical=\{\varnothing\}\cup\bigcup_{r\geq1}\NN^r
\]
denote the set of indices and define
\[
\dep(\ww{s})=r,
\qquad
\wt(\ww{s})=s_1+\cdots+s_r,
\qquad
\dep(\varnothing)=\wt(\varnothing)=0
\]
for every $\ww{s}=(s_1,\ldots,s_r)\in\Ical$.
For an admissible non-empty index
$\ww{s}=(s_1,\ldots,s_r)\in\Ical$, namely $s_1\geq2$, the classical
\emph{multiple zeta value} is defined by
\[
\zeta(\ww{s})
=
\sum_{n_1>\cdots>n_r\geq1}
\frac{1}{n_1^{s_1}\cdots n_r^{s_r}}
\in\RR.
\]
Multiple zeta values are known to satisfy the stuffle and shuffle product
relations.  A further structure appears at the motivic level. Following
\cite{Bro12}, let
$\Zcal^{\mfk}$ denote the graded $\QQ$-algebra of \emph{motivic multiple zeta
values}, and write $\zeta^{\mfk}(\ww{s})\in\Zcal^{\mfk}$
for the motivic multiple zeta value associated with an admissible index
$\ww{s}$.  By \cite{Bro12}, the quotient
\[
\ov{\Zcal}^{\mfk}
=
\Zcal^{\mfk}/\zeta^{\mfk}(2)\Zcal^{\mfk}
\]
carries a natural structure of a graded Hopf algebra, and
$\Zcal^{\mfk}$ is naturally a graded comodule over
$\ov{\Zcal}^{\mfk}$ through the motivic coaction
\[
\Delta^{\mfk}:
\Zcal^{\mfk}
\longrightarrow
\ov{\Zcal}^{\mfk}\otimes_\QQ\Zcal^{\mfk}.
\]

Let $\Zcal$ denote the $\QQ$-algebra generated by the classical multiple
zeta values. There is a surjective period homomorphism
\[
\operatorname{per}:
\Zcal^{\mfk}
\longrightarrow
\Zcal
\]
satisfying
\[
\operatorname{per}
\left(\zeta^{\mfk}(\ww{s})\right)
=
\zeta(\ww{s}).
\]
It is conjectured that this period map is an isomorphism \cite[\S3.2]{Bro14}.  This naturally
raises the question of whether the motivic coaction admits a counterpart
directly at the level of multiple zeta values, in particular after
passing to the quotient by $\zeta(2)$.
\subsection{\texorpdfstring{$\infty$}{∞}-adic Multiple Zeta Values}
We now turn to the function field setting. Let
$p$ be a prime number and let $q$ be a power of $p$. Put
\[
A=\FF_q[\theta],
\qquad
K=\FF_q(\theta),
\]
and let $K_\infty=\FF_q(\!(1/\theta)\!)$ be the completion of $K$ at the
infinite place. Let $\CC_\infty$ be the completion of an algebraic
closure of $K_\infty$, and view a fixed algebraic closure $\ov{K}$ of
$K$ as a subfield of $\CC_\infty$. 

Thakur introduced the function field analogue of multiple zeta values
$\zeta_A(\ww{s})$ in \cite{Tha04} for $\ww{s}\in\Ical$. Throughout this paper, we refer to
these values as the \emph{$\infty$-adic multiple zeta values}. In \cite{Cha14}, Chang introduced the \emph{Carlitz multiple polylogarithms} and expressed the $\infty$-adic multiple zeta values as $K$-linear combinations of their special values at integral points, extending the work of Anderson and Thakur \cite{AT90} when $\dep(\ww{s})=1$. In what follows, we will mainly consider the special values $\Li_{\ww{s}}(\mathbf{1})$, where $\ww{s}\in\Ical$ and $\mathbf{1}=(1,\ldots,1)$ denotes the tuple of length equal to $\dep(\ww{s})$. For the precise definitions, we refer to \S\ref{sec:algebraic-setup}.

Put $L_1=\theta-\theta^q$ and fix an $\FF_p(L_1)$-subalgebra $R$ of
$\ov{K}$. We let $\Zcal_{\infty,R}$ denote the $R$-algebra generated by
the $\infty$-adic multiple zeta values $\zeta_A(\ww{s})$, with
$\ww{s}\in\Ical$. For each $w\geq0$, we also write
\[
\Zcal_{\infty,R,w}
=
\Span_R
\left\{
\zeta_A(\ww{s})
:
\ww{s}\in\Ical,\ \wt(\ww{s})=w
\right\}.
\]
Recently, Mishiba \cite{Mis26} introduced the
$\ast$-inverse values
$\zeta_A^{\ant}(\ww{s})$ and
$\Li_{\ww{s}}^\ant(\mathbf{1})$
associated with $\zeta_A(\ww{s})$ and
$\Li_{\ww{s}}(\mathbf{1})$, respectively.
For every $\ww{s}\in\Ical$, we have
$\Li_{\ww{s}}(\mathbf{1}),\zeta_A^{\ant}(\ww{s}),\Li_{\ww{s}}^\ant(\mathbf{1})
\in \Zcal_{\infty,R,\wt(\ww{s})}$.
We further put
\[
\ov{\Zcal}_{\infty,R}
=
\Zcal_{\infty,R}/\zeta_A(q-1)\Zcal_{\infty,R},
\]
and let
\[
\pi_\infty:
\Zcal_{\infty,R}
\longrightarrow
\ov{\Zcal}_{\infty,R}
\]
denote the quotient map. For $\ww{s}=(s_1,\ldots,s_r)\in\Ical$ and $0\leq i\leq r$, we define
\[
\ww{s}[:i]=(s_1,\ldots,s_i),
\qquad
\ww{s}[i+1:]=(s_{i+1},\ldots,s_r).
\]
Any truncation outside the indicated range is understood to be
$\varnothing$. 

In \cite{Mis26}, Mishiba constructed a right coaction of $\ov{\Zcal}_{\infty,R}$ on $\Zcal_{\infty,R}$.
The resulting structure may be summarized as follows.

\begin{thm}[{\cite[Theorem~1.3.4]{Mis26}}]\label{thm:mishiba-coaction}
The following assertions hold.

\begin{enumerate}
\item
There exists an $R$-algebra homomorphism
\[
\widetilde{\Delta}_\infty^{\Li}:
\Zcal_{\infty,R}
\longrightarrow
\Zcal_{\infty,R}\otimes_R\ov{\Zcal}_{\infty,R}
\]
such that, for every
$\ww{s}\in\Ical$ with $\dep(\ww{s})=r$,
\[
\widetilde{\Delta}_\infty^{\Li}
\left(
\Li_{\ww{s}}(\mathbf{1})
\right)
=
\sum_{i=0}^r
\Li_{\ww{s}[:i]}(\mathbf{1})
\otimes
\pi_\infty
(\Li_{\ww{s}[i+1:]}(\mathbf{1})).
\]

\item The map $\widetilde{\Delta}_\infty^{\Li}$ is a coaction, and
$\ov{\Zcal}_{\infty,R}$ carries a connected graded Hopf algebra
structure compatible with this coaction.

\item
The antipode
\[
S_\infty:
\ov{\Zcal}_{\infty,R}
\longrightarrow
\ov{\Zcal}_{\infty,R}
\]
is a non-trivial $R$-algebra involution satisfying
\[
S_\infty
\left(
\pi_\infty(\Li_{\ww{s}}(\mathbf{1}))
\right)
=
\pi_\infty
\left(
\Li_{\ww{s}}^\ant(\mathbf{1})
\right)
\]
for every $\ww{s}\in\Ical$. 
\end{enumerate}
\end{thm}

Mishiba further conjectured that the same maps are compatible with the
$\infty$-adic multiple zeta values.

\begin{conj}[{\cite[Conjecture~1.3.7]{Mis26}}]\label{conj:mishiba}
For every
$\ww{s}=(s_1,\ldots,s_r)\in\Ical$, the following assertions hold.

\begin{enumerate}
\item
The coaction $\widetilde{\Delta}_\infty^{\Li}$ satisfies
\[
\widetilde{\Delta}_\infty^{\Li}
\left(
\zeta_A(\ww{s})
\right)
=
\sum_{i=0}^r
\zeta_A(\ww{s}[:i])
\otimes
\pi_\infty
\left(
\zeta_A(\ww{s}[i+1:])
\right).
\]

\item
The antipode $S_\infty$ satisfies
\[
S_\infty
\left(
\pi_\infty(\zeta_A(\ww{s}))
\right)
=
\pi_\infty
\left(
\zeta_A^\ant(\ww{s})
\right).
\]
\end{enumerate}
\end{conj}

\subsection{Main Results}
We develop a formalism based on evaluation maps and deconcatenation
which treats the $\zeta$- and $\Li$-realizations, together with their
$\ast$-inverse counterparts, in a uniform manner. Applying this
formalism to $\infty$-adic multiple zeta values, we construct a
coaction directly from the $\zeta$-realization and show that it
coincides with Mishiba's coaction $\widetilde{\Delta}_\infty^\Li$.
The same formalism also determines the coaction on the corresponding
$\ast$-inverse values and the action of the antipode.

Our main results are the following.

\begin{thm}[Restated as \cref{thm:infty-coaction}]
\label{thm:main-infty-coaction}
Let $R$ be an $\FF_p(L_1)$-subalgebra of $\ov{K}$. There exists a
well-defined $R$-algebra homomorphism
\[
\widetilde{\Delta}_\infty:
\Zcal_{\infty,R}
\longrightarrow
\Zcal_{\infty,R}\otimes_R\ov{\Zcal}_{\infty,R}
\]
such that, for every
$\ww{s}\in\Ical$ with $\dep(\ww{s})=r$,
\[
\widetilde{\Delta}_\infty
(
\zeta_A(\ww{s})
)
=
\sum_{i=0}^r
\zeta_A(\ww{s}[:i])
\otimes
\pi_\infty
(
\zeta_A(\ww{s}[i+1:])
).
\]
Moreover, we have
\[
\widetilde{\Delta}_\infty
(
\Li_{\ww{s}}(\mathbf{1})
)
=
\sum_{i=0}^r
\Li_{\ww{s}[:i]}(\mathbf{1})
\otimes
\pi_\infty
(
\Li_{\ww{s}[i+1:]}(\mathbf{1})
),
\]
\[
\widetilde{\Delta}_\infty
(
\zeta_A^\ant(\ww{s})
)
=
\sum_{i=0}^r
\zeta_A^\ant(\ww{s}[i+1:])
\otimes
\pi_\infty
(
\zeta_A^\ant(\ww{s}[:i])
),
\]
and
\[
\widetilde{\Delta}_\infty
(
\Li_{\ww{s}}^\ant(\mathbf{1})
)
=
\sum_{i=0}^r
\Li_{\ww{s}[i+1:]}^\ant(\mathbf{1})
\otimes
\pi_\infty
(
\Li_{\ww{s}[:i]}^\ant(\mathbf{1})
).
\]
In particular, we have
\[
\widetilde{\Delta}_\infty
=
\widetilde{\Delta}_\infty^\Li
\]
and \cref{conj:mishiba}~(i) holds.
\end{thm}

Thus, after passing to the quotient by the ideal $\zeta_A(q-1)$, the usual deconcatenation of indices is compatible with the $q$-shuffle product and gives the coproduct on $\ov{\Zcal}_{\infty,R}$. 

The preceding theorem also determines the action of the antipode on
the $\zeta$-realization.

\begin{thm}[Restated as \cref{thm:infty-Hopf-structure}]
\label{thm:main-infty-antipode}
The antipode
\[
S_\infty:
\ov{\Zcal}_{\infty,R}
\longrightarrow
\ov{\Zcal}_{\infty,R}
\]
satisfies, for every $\ww{s}\in\Ical$,
\[
S_\infty
(
\pi_\infty(\zeta_A(\ww{s}))
)
=
\pi_\infty
(
\zeta_A^\ant(\ww{s})
),\qquad
S_\infty
(
\pi_\infty(\Li_{\ww{s}}(\mathbf{1}))
)
=
\pi_\infty
(
\Li_{\ww{s}}^\ant(\mathbf{1})
).
\]
In particular, \cref{conj:mishiba}~(ii) holds.
\end{thm}

Consequently, modulo $\zeta_A(q-1)$, the $\ast$-inverse values satisfy exactly the same algebraic relations as the ordinary $\infty$-adic multiple zeta values.

The rest of the paper is organized as follows.
In \S\ref{sec:preliminary}, we recall the necessary background on the
$\zeta$- and $\Li$-realizations, their product structures, and the
linear relations among the corresponding values.
In \S\ref{sec:formalism}, we first compare the ordinary and
$\ast$-inverse evaluations. We then study the behavior of the
relations $\sA^\bullet(\ww{s};m;\ww{n})$ under deconcatenation and use
these results to construct a formal coaction.
Finally, in \S\ref{sec:infty-coaction}, we apply this formalism to
$\infty$-adic multiple zeta values, identify the resulting coaction
with Mishiba's coaction arising from the $\Li$-realization, determine
the coaction on the $\ast$-inverse values, and show that the antipode
on $\ov{\Zcal}_{\infty,R}$ is given by the $\ast$-inverse operation.

\section{Preliminaries}\label{sec:preliminary}

\subsection{Algebraic Setup}\label{sec:algebraic-setup}
Recall that
\[
\Ical=\{\varnothing\}\cup\bigcup_{r\geq1}\NN^r
\]
denotes the set of indices. We let 
\[
\Hfk_R=\Span_R\{[\ww{s}]:\ww{s}\in\Ical\}
\] be the free $R$-module generated by the set of indices $\Ical$ and
\[
\Hfk_{R,w}=\Span_R\{[\ww{s}]:\ww{s}\in\Ical,\ \wt(\ww{s})=w\}
\]
for all $w\geq 0$. We put $1=[\varnothing]$.
Then the concatenation map
\[
[-,-]:\Hfk_R\times\Hfk_R\rightarrow\Hfk_R
\]
is the unique $R$-bilinear map determined by
\[
[\ww{s},\ww{n}]
=[s_1,\ldots,s_r,n_1,\ldots,n_\ell],
\qquad
[\varnothing,\ww{s}]=[\ww{s},\varnothing]=[\ww{s}],\qquad [\varnothing,\varnothing]=[\varnothing],
\]
where $\ww{s}=(s_1,\ldots,s_r),\ww{n}=(n_1,\ldots,n_\ell)\in\Ical$ are non-empty indices.
Iterated concatenations are understood $R$-multilinearly and are denoted by
\[
[-,\ldots,-]\colon \Hfk_R\times\cdots\times\Hfk_R\longrightarrow\Hfk_R.
\]
We identify an index with its corresponding basis element when no
confusion can arise. 

On the other hand, the deconcatenation map is the unique $R$-linear map
\[
\Delta:\Hfk_R\longrightarrow\Hfk_R\otimes_R\Hfk_R
\]
such that
\begin{equation}\label{eq:Delta}
   \Delta([\ww{s}])
=\sum_{i=0}^r[\ww{s}[:i]]\otimes[\ww{s}[i+1:]] 
\end{equation}
for all indices $\ww{s}\in\Ical$.

Following \cite{CCM23,Mis26}, we define the realization maps as follows. For $d\geq0$ and $s\in\NN$, set
\[
S_d(s)
=
\sum_{\substack{a\in A\text{ monic}\\ \deg a=d}}
\frac{1}{a^s},\qquad
L_0=1,
\qquad
L_d=\prod_{i=1}^d(\theta-\theta^{q^i})
\quad(d\geq1).
\]
For $\bullet\in\{\zeta,\Li\}$, put
\[
\mathscr{S}_d^\zeta(s)=S_d(s),
\qquad
\mathscr{S}_d^\Li(s)=\frac{1}{L_d^s}.
\]

\begin{df}\label{df:infty-adic-realizations}
Let $\bullet\in\{\zeta,\Li\}$. We define the $R$-linear maps $\mathscr{L}_\infty^\bullet,\mathscr{L}_\infty^{\bullet,\ant}:\Hfk_R\to \CC_\infty$ as follows.   For any non-empty index $\ww{s}=(s_1,\ldots,s_r)\in\Ical$, define
\begin{align*}
\mathscr{L}_\infty^\bullet([\ww{s}])
&=
\sum_{d_1>\cdots>d_r\geq0}
\prod_{i=1}^r\mathscr{S}_{d_i}^\bullet(s_i)\in \CC_\infty,\\
\mathscr{L}_\infty^{\bullet,\ant}([\ww{s}])
&=
(-1)^r
\sum_{0\leq d_1\leq\cdots\leq d_r}
\prod_{i=1}^r\mathscr{S}_{d_i}^\bullet(s_i)\in \CC_\infty.
\end{align*}
Both maps take $1=[\varnothing]$ to $1$ and are extended $R$-linearly
to $\Hfk_R$.
\end{df}

For every $\ww{s}\in\Ical$, we have
\[
\mathscr{L}_\infty^\zeta([\ww{s}])=\zeta_A(\ww{s}),
\qquad
\mathscr{L}_\infty^{\zeta,\ant}([\ww{s}])=\zeta_A^\ant(\ww{s}),
\]
and
\[
\mathscr{L}_\infty^\Li([\ww{s}])=\Li_{\ww{s}}(\mathbf{1}),
\qquad
\mathscr{L}_\infty^{\Li,\ant}([\ww{s}])
=
\Li_{\ww{s}}^\ant(\mathbf{1}).
\]

\subsection{Shuffle and Stuffle Products}\label{sec:shuffle-stuffle}
We next recall the products corresponding to these two realizations.
For the $\zeta$-realization, the existence of the $q$-shuffle
relations was established by Thakur \cite{Tha10}. Chen subsequently
obtained an explicit formula in depth one \cite{Che15}. Based on
Chen's formula, Yamamoto observed the recursive construction of the
$q$-shuffle product, which was later formulated and proved by Shi
(see \cite[Definition~3.1.3 and Theorem~3.1.4]{Shi18}). For the
$\Li$-realization, the corresponding product is the usual
stuffle product \cite{Cha14}.

For $s,n,j\in\NN$, set
\[
\Delta_{s,n}^{[j]}
=
\begin{cases}
\displaystyle
(-1)^{s-1}\binom{j-1}{s-1}
+(-1)^{n-1}\binom{j-1}{n-1},
&q-1\mid j,\\[6pt]
0,&\text{otherwise}.
\end{cases}
\]

For any $\ww{s}=(s_1,\ldots,s_r)\in\Ical$, we put
\[
\ww{s}^{+}=(s_1,\ldots,s_{r-1}),
\qquad \varnothing^+=\varnothing,\qquad
\ww{s}^{-}=(s_2,\ldots,s_r), \qquad \varnothing^-=\varnothing.
\]

\begin{df}\label{df:bullet-stuffle}
Let $\bullet\in\{\zeta,\Li\}$. We define two $R$-bilinear products on $\Hfk_R$, the $q$-shuffle product
$\ast^\zeta$ and the stuffle product $\ast^\Li$, by
\[
1\ast^\bullet P=P\ast^\bullet1=P
\]
and, for non-empty indices
$\ww{s}=(s_1,\ldots,s_r)$ and
$\ww{n}=(n_1,\ldots,n_\ell)$, by
\begin{equation}\label{eq:bullet-stuffle}
\begin{aligned}
\ww{s}\ast^\bullet\ww{n}
={}&[s_1,\ww{s}^-\ast^\bullet\ww{n}]
+[n_1,\ww{s}\ast^\bullet\ww{n}^-]
+[s_1+n_1,\ww{s}^-\ast^\bullet\ww{n}^-]
+D_{\ww{s}}^\bullet(\ww{n}).
\end{aligned}
\end{equation}
Here $D_{\ww{s}}^\Li(\ww{n})=0$, while
\begin{equation}\label{eq:D-zeta-definition}
D_{\ww{s}}^\zeta(\ww{n})
=
\sum_{j=1}^{s_1+n_1-1}
\Delta_{s_1,n_1}^{[j]}
[s_1+n_1-j,
[j]\ast^\zeta(\ww{s}^-\ast^\zeta\ww{n}^-)].
\end{equation}
For every non-empty $\ww{s}\in\Ical$, we set
$D_{\ww{s}}^\bullet(1)=0$ and extend
$D_{\ww{s}}^\bullet(-)$ to $\Hfk_R$ $R$-linearly. We abbreviate
\[
D_c^\bullet(-)=D_{(c)}^\bullet(-),
\qquad
D_{\ww{s}}(-)=D_{\ww{s}}^\zeta(-).
\]
\end{df}

Thus \eqref{eq:bullet-stuffle} specializes to the $q$-shuffle product
of $\infty$-adic multiple zeta values when $\bullet=\zeta$, and
to the usual stuffle product when
$\bullet=\Li$. Then we have the following theorem.

\begin{thm}[{\cite{Cha14,Shi18}; see also \cite[Proposition~2.7]{CCM23}}]
\label{thm:strict-bullet-character}
Let $\bullet\in\{\zeta,\Li\}$. For any $P,Q\in\Hfk_R$, we have

\[
\mathscr{L}_\infty^\bullet(P\ast^\bullet Q)
=
\mathscr{L}_\infty^\bullet(P)\mathscr{L}_\infty^\bullet(Q).
\]
\end{thm}

\subsection{Linear Relations}\label{sec:linear-relations}
By \cite[Theorem~2.2.1]{Cha14}, it is known that $\Zcal_{\infty,R}$ forms a $\ZZ_{\geq 0}$-graded algebra and every $\ov{K}$-linear relation among $\zeta_A(\ww{s})$ comes from a $K$-linear relation. 
Todd first conjectured a formula for the dimension of
$\Zcal_{\infty,K,w}$ in each weight $w$
\cite[Conjecture~7.1]{Tod18}. Later, Thakur further conjectured that a basis of $\Zcal_{\infty,K,w}$
is given by the values $\zeta_A(\ww{s})$ with $\ww{s}$ ranging over the set
of Thakur indices of weight $w$ 
\cite[Conjecture~8.2]{Tha17},
\[
\Ical_w^{\mathrm{T}}
=
\begin{cases}
\{\varnothing\},
& w=0,\\[1mm]
\left\{
(s_1,\ldots,s_r)\in\Ical :
s_i\leq q\ (1\leq i<r), 
s_r<q, \wt(\ww{s})=w
\right\},
& w\geq1.
\end{cases}
\]
We remark that $\zeta_A(\ww{s})=\mathscr{L}_\infty^\zeta([\ww{s}])=\mathscr{L}_\infty^\Li([\ww{s}])=\Li_{\ww{s}}(\mathbf{1})$ for all $\ww{s}\in\Ical^{\mathrm{T}}=\bigcup_{w\geq 0}\Ical_w^{\mathrm{T}}$ (see \cite[(2.7)]{CCM23}).
In \cite[Theorem~A]{ND21}, Ngo Dac proved that every multiple zeta value is an
$\FF_p[L_1]$-linear combination of these values. The full basis conjecture was later
proved independently in \cite[Theorem~1.5]{CCM23} and
\cite[Theorem~B]{IKLNP24}. Precisely, we have the following theorem.

\begin{thm}[{\cite[Theorem~1.5 and Corollary~3.7]{CCM23}, \cite[Theorem~B]{IKLNP24};
see also \cite[\S1.2]{Mis26}}]\label{thm:basis_for_infty}
Let $\bullet\in\{\zeta,\Li\}$. Then for each $w\geq0$,
\[
\Zcal_{\infty,R,w}
=
\mathscr{L}_\infty^\bullet(\Hfk_{R,w}),
\]
and
\[
\{
\mathscr{L}_\infty^\bullet([\ww{s}])
:
\ww{s}\in\Ical_w^{\mathrm{T}}
\}
\]
forms an $R$-basis of $\Zcal_{\infty,R,w}$.
\end{thm}

In fact, the linear relations among these values are well understood from \cite{CCM23} (see also \cite{IKLNP24}),
and Mishiba considered the same framework over an arbitrary
$\FF_p[L_1]$-subalgebra $R$ of $\ov{K}$ (see \cite[\S2.3]{Mis26}). We recall the formulation below. Notice that we assume $R$ is an $\FF_p(L_1)$-algebra (see \cref{rmk:F_p[L_1]doesnotwork,rmk:F_p[L_1]doesnotwork-antipode}).

\begin{df}[{\cite[\S3]{CCM23}; see also \cite[\S2.3]{Mis26}}]\label{df:alpha}
For $c\in\NN$, $\ww{s}\in\Ical$, and
$\bullet\in\{\zeta,\Li\}$, define the $R$-linear map
\[
\alpha_{c;\ww{s}}^\bullet:\Hfk_R\longrightarrow\Hfk_R,
\qquad
\alpha_{c;\ww{s}}^\bullet(P)
=
[c,\ww{s}\ast^\bullet P].
\]
For $m\geq 1$, put
\[
\alpha_{c;\ww{s}}^{\bullet,m}
=
\underbrace{
\alpha_{c;\ww{s}}^\bullet\circ\cdots\circ
\alpha_{c;\ww{s}}^\bullet
}_{m\text{ times}},
\qquad
\alpha_{c;\ww{s}}^{\bullet,0}
=
\id_{\Hfk_R}.
\]
We also define the $R$-bilinear map
\[
\boxplus:\Hfk_R\times\Hfk_R\longrightarrow\Hfk_R
\]
by
\[
1\boxplus P=P\boxplus1=0
\qquad\text{and}\qquad
\ww{s}\boxplus\ww{n}
=
[s_1,\ldots,s_{r-1},s_r+n_1,n_2,\ldots,n_\ell]
\]
for any $P\in\Hfk_R$ and non-empty indices
$\ww{s}=(s_1,\ldots,s_r),\ww{n}=(n_1,\ldots,n_\ell)\in\Ical$.
Here $0$ denotes the zero vector of $\Hfk_R$, rather than the empty word
$1=[\varnothing]$.
\end{df}

For $c\in\NN$ and $m\geq1$, we write $\{c\}^m=(c,\ldots,c)$ for the index consisting of $m$ copies of $c$ and $\{c\}^0=\varnothing$.

\begin{df}[{\cite[\S2.3]{Mis26}; see also \cite[\S3]{CCM23}}]\label{df:Mishiba-relations}
Let $\bullet\in\{\zeta,\Li\}$. Put $\varepsilon_\zeta=1$ and $\varepsilon_\Li=0$. For $\ww{s},\ww{n}\in\Ical$ and $m\geq1$, define
\[
\begin{aligned}
\sA^\bullet(\ww{s};m;\ww{n})
={}&
[\ww{s},\{q\}^m,\ww{n}]
+
[\ww{s},\{q\}^m\boxplus\ww{n}]
+
\varepsilon_\bullet
[\ww{s},\{q\}^{m-1},D_q(\ww{n})]\\
&-
L_1^m
[\ww{s},\alpha_{1;(q-1)}^{\bullet,m}(\ww{n})]
-
L_1^m
[\ww{s}]\boxplus\alpha_{1;(q-1)}^{\bullet,m}(\ww{n})\\
&-
\varepsilon_\bullet L_1^m
[\ww{s}^+,D_{s_r}(\alpha_{1;(q-1)}^{\bullet,m}(\ww{n}))],
\end{aligned}
\]
if $\ww{s}=(s_1,\ldots,s_r)$ is non-empty, and
\[
\begin{aligned}
\sA^\bullet(\varnothing;m;\ww{n})
={}&
[\{q\}^m,\ww{n}]
+
[\{q\}^m\boxplus\ww{n}]
+
\varepsilon_\bullet
[\{q\}^{m-1},D_q(\ww{n})]
-
L_1^m
\alpha_{1;(q-1)}^{\bullet,m}(\ww{n}).
\end{aligned}
\]
We also set
\[
\mathscr{R}_R^\bullet
=
\Span_R
\left\{
\sA^\bullet(\ww{s};m;\ww{n})
:
\ww{s},\ww{n}\in\Ical,\ m\geq1
\right\}
\subseteq\Hfk_R
\]
and
\[
\mathscr{R}_{R,w}^\bullet
=
\Span_R
\left\{
\sA^\bullet(\ww{s};m;\ww{n})
:
\ww{s},\ww{n}\in\Ical,\ m\geq1,\ \wt(\ww{s})+mq+\wt(\ww{n})=w
\right\}
\subseteq\Hfk_{R,w}
\]
\end{df}
Then we have the following proposition.

\begin{prop}[{\cite[\S3]{CCM23}; see also \cite[Proposition~2.3.2]{Mis26}}]\label{prop:Mishiba-kernel}
For every $\bullet\in\{\zeta,\Li\}$,
$\ww{s},\ww{n}\in\Ical$, and $m\geq1$, we have
\[
\mathscr{L}_\infty^\bullet
\left(
\sA^\bullet(\ww{s};m;\ww{n})
\right)
=
0.
\]
Moreover, for all $w\geq 0$,
\[
\ker
\left(
\mathscr{L}_\infty^\bullet|_{\Hfk_{R,w}}:
\Hfk_{R,w}\longrightarrow\Zcal_{\infty,R}
\right)
=
\mathscr{R}_{R,w}^\bullet,\qquad
\ker
\left(
\mathscr{L}_\infty^\bullet:
\Hfk_R\longrightarrow\Zcal_{\infty,R}
\right)
=
\mathscr{R}_R^\bullet.
\]
\end{prop}

\section{Evaluation Maps and Deconcatenation}\label{sec:formalism}
In this section, we study evaluation maps together with the
deconcatenation map. We first compare the ordinary and
$\ast$-inverse evaluations. We then study the behavior of the
relations $\sA^\bullet(\ww{s};m;\ww{n})$ under deconcatenation.
These results will be used in the next section to construct the coaction on $\infty$-adic multiple zeta values.
\subsection{Evaluation Maps and the \texorpdfstring{$\ast$}{star}-Inverse}\label{sec:evaluation-deconcatenation}
Throughout this section, let $\bullet\in\{\zeta,\Li\}$, let $B$ be
a commutative associative unital $R$-algebra, and let
$\ev,\ev^\ant:\Hfk_R\to B$ be $R$-linear maps. Put
$\varepsilon_\zeta=1$ and $\varepsilon_\Li=0$. We consider the
following conditions:
\begin{enumerate}[label=\textup{(C\arabic*)},
ref=\textup{(C\arabic*)}]
\item\label{cond:unital}
$\ev(1)=\ev^\ant(1)=1$.

\item\label{cond:convolution}
For every non-empty index $\ww{s}\in\Ical$,
\[
m_B(\ev\otimes\ev^\ant)\Delta([\ww{s}])
=
0
=
m_B(\ev^\ant\otimes\ev)\Delta([\ww{s}]),
\]
where $m_B$ denotes the multiplication map of $B$ and $\Delta$ is defined in \eqref{eq:Delta}.
\end{enumerate}
We also consider the following conditions, which depend on the choice of $\bullet\in\{\zeta,\Li\}$.
\begin{enumerate}[label=\ensuremath{(\mathrm{C}\arabic*)^\bullet},
ref=\ensuremath{(\mathrm{C}\arabic*)^\bullet},start=3]
\item\label{cond:vanishing}
We have $\ev^\ant([j])=0$ whenever $q-1\mid j$
if $\bullet=\zeta$, and $\ev^\ant([q-1])=0$
if $\bullet=\Li$.

\item\label{cond:ev-product}
For all $P,Q\in\Hfk_R$,
\[
\ev(P\ast^\bullet Q)=\ev(P)\ev(Q).
\]

\item\label{cond:ant-product}
For all $P,Q\in\Hfk_R$,
\[
\ev^\ant(P\ast^\bullet Q)=\ev^\ant(P)\ev^\ant(Q).
\]
\end{enumerate}

We then define
\[
\Rcal=(\id\otimes\ev^\ant)\Delta:
\Hfk_R\longrightarrow\Hfk_R\otimes_RB
\]
and equip $\Hfk_R\otimes_RB$ with the product
\[
(P\otimes x)\ast^\bullet(Q\otimes y)
=(P\ast^\bullet Q)\otimes xy
\]
for any $P,Q\in\Hfk_R$ and $x,y\in B$.
\begin{lem}\label{lem:R-concatenation}
Let $\ww{s}=(s_1,\ldots,s_r)\in\Ical$ be a non-empty index and
$P\in\Hfk_R$. Then
\[
\begin{aligned}
\Rcal([\ww{s},P])
={}&
1\otimes\ev^\ant([\ww{s},P])+
\sum_{i=1}^{r-1}
[\ww{s}[:i]]
\otimes
\ev^\ant([\ww{s}[i+1:],P])+
([\ww{s},-]\otimes\id)\Rcal(P),
\end{aligned}
\]
where $[\ww{s},-]$ denotes left concatenation by $\ww{s}$.
In particular, if $\dep(\ww{s})=1$, say $\ww{s}=(s)$, then
\begin{equation}\label{eq:R-left-letter}
\Rcal([s,P])
=
1\otimes\ev^\ant([s,P])
+
([s,-]\otimes\id)\Rcal(P).
\end{equation}
\end{lem}
\begin{proof}
By $R$-linearity, it suffices to prove the assertion for
$P=[\ww{n}]$, where $\ww{n}\in\Ical$. The result follows from the fact that
\[
\begin{aligned}
\Rcal([\ww{s},\ww{n}])
={}&
1\otimes\ev^\ant([\ww{s},\ww{n}])+
\sum_{i=1}^{r-1}
[\ww{s}[:i]]
\otimes
\ev^\ant([\ww{s}[i+1:],\ww{n}])\\&\hspace{44mm}+
\sum_{\ww{n}=[A,B]}
[\ww{s},A]\otimes\ev^\ant(B)
\end{aligned}
\]
and
$\Rcal(\ww{n})
=
\sum_{\ww{n}=[A,B]}
A\otimes\ev^\ant(B)$.
\end{proof}

\begin{lem}\label{lem:bullet-stuffle-transfer}
Let $\bullet\in\{\zeta,\Li\}$. Assume Conditions
\ref{cond:unital}, \ref{cond:convolution},
and \ref{cond:ev-product}.
If $\bullet=\zeta$, we further assume Condition \ref{cond:vanishing}.
Then for all $P,Q\in\Hfk_R$,
\[
\Rcal(P\ast^\bullet Q)
=
\Rcal(P)\ast^\bullet\Rcal(Q).
\]
In particular, Condition \ref{cond:ant-product} holds.
\end{lem}

\begin{proof}
By bilinearity, it is enough to take $P=[\ww{s}]$ and $Q=[\ww{n}]$,
where $\ww{s},\ww{n}\in\Ical$. We use induction on
$\dep(\ww{s})+\dep(\ww{n})$. If one of the indices is empty, the assertion follows immediately
from Condition \ref{cond:unital}. For non-empty $\ww{s},\ww{n}\in\Ical$, write
$
[\ww{s}]=[s_1,\ww{s}^-]$ and 
$[\ww{n}]=[n_1,\ww{n}^-]
$. 
Recall that
\begin{align*}
\ww{s}\ast^\bullet\ww{n}
={}&
[s_1,\ww{s}^-\ast^\bullet\ww{n}]+
[n_1,\ww{s}\ast^\bullet\ww{n}^-]+
[s_1+n_1,\ww{s}^-\ast^\bullet\ww{n}^-]
\\&+
\varepsilon_\bullet
\sum_{1\leq j<s_1+n_1}
\Delta_{s_1,n_1}^{[j]}
[s_1+n_1-j,
[j]\ast^\zeta(\ww{s}^-\ast^\zeta\ww{n}^-)].
\end{align*}
Write
\[
\Rcal(\ww{s}^-)
=
\sum_{\ww{s}^-=[A,B]}A\otimes\ev^\ant(B),
\qquad
\Rcal(\ww{n}^-)
=
\sum_{\ww{n}^-=[C,D]}C\otimes\ev^\ant(D).
\]
Then
\[
\Rcal(\ww{s})
=
1\otimes\ev^\ant(\ww{s})
+
\sum_{\ww{s}^-=[A,B]}
[s_1,A]\otimes\ev^\ant(B),
\]
and similarly
\[
\Rcal(\ww{n})
=
1\otimes\ev^\ant(\ww{n})
+
\sum_{\ww{n}^-=[C,D]}
[n_1,C]\otimes\ev^\ant(D).
\]
Hence,
$\Rcal(\ww{s})\ast^\bullet\Rcal(\ww{n})$
consists of the term
\begin{equation}\label{eq:idconstant}
1\otimes \ev^{\ant}(\ww{s})\ev^{\ant}(\ww{n}),
\end{equation}
the two cross terms
\begin{equation}\label{eq:cross-s}
\sum_{\ww{s}^-=[A,B]}
[s_1,A]\otimes
\ev^\ant(B)\ev^\ant(\ww{n}),
\end{equation}
and
\begin{equation}\label{eq:cross-n}
\sum_{\ww{n}^-=[C,D]}
[n_1,C]\otimes
\ev^\ant(\ww{s})\ev^\ant(D),
\end{equation}
together with
\begin{equation}\label{eq:prefix-product}
\sum_{\substack{\ww{s}^-=[A,B]\\\ww{n}^-=[C,D]}}
([s_1,A]\ast^\bullet[n_1,C])
\otimes
\ev^\ant(B)\ev^\ant(D).
\end{equation}

Applying the recursion to the first tensor factor in
\eqref{eq:prefix-product}, it becomes the sum of the following three terms:
\begin{align}
&\sum_{\substack{\ww{s}^-=[A,B]\\\ww{n}^-=[C,D]}}
[s_1,A\ast^\bullet[n_1,C]]
\otimes\ev^\ant(B)\ev^\ant(D),
\label{eq:prefix-first}\\
&\sum_{\substack{\ww{s}^-=[A,B]\\\ww{n}^-=[C,D]}}
[n_1,[s_1,A]\ast^\bullet C]
\otimes\ev^\ant(B)\ev^\ant(D),
\label{eq:prefix-second}\\
&\sum_{\substack{\ww{s}^-=[A,B]\\\ww{n}^-=[C,D]}}
[s_1+n_1,A\ast^\bullet C]
\otimes\ev^\ant(B)\ev^\ant(D),
\label{eq:prefix-third}
\end{align}
together with 
\begin{equation}\label{eq:prefix-correction}
\begin{aligned}
\varepsilon_\bullet
\sum_{1\leq j<s_1+n_1}
\Delta_{s_1,n_1}^{[j]}
\sum_{\substack{\ww{s}^-=[A,B]\\\ww{n}^-=[C,D]}}
&[s_1+n_1-j,
[j]\ast^\zeta(A\ast^\zeta C)]\otimes\ev^\ant(B)\ev^\ant(D).
\end{aligned}
\end{equation}
Notice that the sum of \eqref{eq:cross-s} and \eqref{eq:prefix-first}
becomes
\[
\begin{aligned}
&\sum_{\ww{s}^-=[A,B]}
[s_1,A]\otimes\ev^\ant(B)\ev^\ant(\ww{n})+
\sum_{\substack{\ww{s}^-=[A,B]\\\ww{n}^-=[C,D]}}
[s_1,A\ast^\bullet[n_1,C]]
\otimes\ev^\ant(B)\ev^\ant(D)\\
&=
([s_1,-]\otimes\id)
\left(
\Rcal(\ww{s}^-)\ast^\bullet\Rcal(\ww{n})
\right)\\
&=
([s_1,-]\otimes\id)
\Rcal(\ww{s}^-\ast^\bullet\ww{n})\\
&=
\Rcal
\left(
[s_1,\ww{s}^-\ast^\bullet\ww{n}]
\right)-1\otimes \ev^{\ant}\left(
[s_1,\ww{s}^-\ast^\bullet\ww{n}]
\right),
\end{aligned}
\]
where the second equality follows from the induction hypothesis and
the last equality follows from \eqref{eq:R-left-letter}. 
Similarly, the sum of \eqref{eq:cross-n} and
\eqref{eq:prefix-second} gives
\[
\begin{aligned}
&\sum_{\ww{n}^-=[C,D]}
[n_1,C]\otimes\ev^\ant(\ww{s})\ev^\ant(D)+
\sum_{\substack{\ww{s}^-=[A,B]\\\ww{n}^-=[C,D]}}
[n_1,[s_1,A]\ast^\bullet C]
\otimes\ev^\ant(B)\ev^\ant(D)\\
&=
([n_1,-]\otimes\id)
\left(
\Rcal(\ww{s})\ast^\bullet\Rcal(\ww{n}^-)
\right)\\
&=
([n_1,-]\otimes\id)
\Rcal(\ww{s}\ast^\bullet\ww{n}^-)\\
&=
\Rcal
\left(
[n_1,\ww{s}\ast^\bullet\ww{n}^-]
\right)-1\otimes \ev^{\ant}\left(
[n_1,\ww{s}\ast^\bullet\ww{n}^-]
\right).
\end{aligned}
\]
For \eqref{eq:prefix-third}, by \eqref{eq:R-left-letter} and the induction hypothesis, we have
\[
\begin{aligned}
&\sum_{\substack{\ww{s}^-=[A,B]\\\ww{n}^-=[C,D]}}
[s_1+n_1,A\ast^\bullet C]
\otimes\ev^\ant(B)\ev^\ant(D)\\
&=
([s_1+n_1,-]\otimes\id)
\left(
\Rcal(\ww{s}^-)\ast^\bullet\Rcal(\ww{n}^-)
\right)\\
&=
([s_1+n_1,-]\otimes\id)
\Rcal(\ww{s}^-\ast^\bullet\ww{n}^-)\\
&=
\Rcal
\left(
[s_1+n_1,\ww{s}^-\ast^\bullet\ww{n}^-]
\right)-1\otimes \ev^{\ant}\left(
[s_1+n_1,\ww{s}^-\ast^\bullet\ww{n}^-]
\right).
\end{aligned}
\]
Now, if $\bullet=\zeta$, note that for every $j$ with $\Delta_{s_1,n_1}^{[j]}\neq0$, we have
$q-1\mid j$. Hence, by Condition \ref{cond:vanishing},
\[
\Rcal([j])
=
1\otimes\ev^\ant([j])+[j]\otimes1
=
[j]\otimes1.
\]
Applying the induction hypothesis gives
\[
\begin{aligned}
&\Rcal\left(
[j]\ast^\zeta(\ww{s}^-\ast^\zeta\ww{n}^-)
\right)\\
&=
\Rcal([j])
\ast^\zeta
\Rcal(\ww{s}^-\ast^\zeta\ww{n}^-)\\
&=
([j]\otimes1)
\ast^\zeta
\left(
\Rcal(\ww{s}^-)\ast^\zeta\Rcal(\ww{n}^-)
\right)\\
&=
\sum_{\substack{\ww{s}^-=[A,B]\\\ww{n}^-=[C,D]}}
\left(
[j]\ast^\zeta(A\ast^\zeta C)
\right)
\otimes\ev^\ant(B)\ev^\ant(D).
\end{aligned}
\]
Therefore
\eqref{eq:prefix-correction} becomes
\[
\begin{aligned}
&\sum_{1\leq j<s_1+n_1}
\Delta_{s_1,n_1}^{[j]}\sum_{\substack{\ww{s}^-=[A,B]\\\ww{n}^-=[C,D]}}
[s_1+n_1-j,
[j]\ast^\zeta(A\ast^\zeta C)]
\otimes\ev^\ant(B)\ev^\ant(D)\\
&=
\sum_{1\leq j<s_1+n_1}
\Delta_{s_1,n_1}^{[j]}([s_1+n_1-j,-]\otimes\id)
\Rcal\left(
[j]\ast^\zeta(\ww{s}^-\ast^\zeta\ww{n}^-)
\right)\\
&=
\Rcal\left(\sum_{1\leq j<s_1+n_1}
\Delta_{s_1,n_1}^{[j]}
[s_1+n_1-j,
[j]\ast^\zeta(\ww{s}^-\ast^\zeta\ww{n}^-)]
\right)\\
&\quad-1\otimes\ev^{\ant}\left(\sum_{1\leq j<s_1+n_1}
\Delta_{s_1,n_1}^{[j]}
[s_1+n_1-j,
[j]\ast^\zeta(\ww{s}^-\ast^\zeta\ww{n}^-)]
\right),
\end{aligned}
\]
where the last equality follows from \eqref{eq:R-left-letter}.

Therefore, for either $\bullet\in\{\zeta,\Li\}$, we see that 
\begin{align}\label{eq:difference}
\Rcal(\ww{s}\ast^\bullet\ww{n})
-
\Rcal(\ww{s})\ast^\bullet\Rcal(\ww{n})
=
1\otimes
\left(
\ev^\ant(\ww{s}\ast^\bullet\ww{n})
-
\ev^\ant(\ww{s})\ev^\ant(\ww{n})
\right).
\end{align}
It remains to show that the right-hand side of \eqref{eq:difference}
vanishes. Apply $m_B(\ev\otimes\id)$ to both sides of
\eqref{eq:difference}. By Condition \ref{cond:convolution},
\[
m_B(\ev\otimes\id)
\Rcal(\ww{s}\ast^\bullet\ww{n})=0.
\]
On the other hand, by Condition \ref{cond:ev-product},
$m_B(\ev\otimes\id)$ is multiplicative with respect to
$\ast^\bullet$. Hence, Condition \ref{cond:convolution} implies that
\[
\begin{aligned}
m_B(\ev\otimes\id)
\left(
\Rcal(\ww{s})\ast^\bullet\Rcal(\ww{n})
\right)=
m_B(\ev\otimes\id)\Rcal(\ww{s})\,
m_B(\ev\otimes\id)\Rcal(\ww{n})
=0.
\end{aligned}
\]
Therefore, applying $m_B(\ev\otimes\id)$ to \eqref{eq:difference}
gives Condition \ref{cond:ant-product}
\[
\ev^\ant(\ww{s}\ast^\bullet\ww{n})
=
\ev^\ant(\ww{s})\ev^\ant(\ww{n}).
\]
Substituting this into \eqref{eq:difference}, we obtain
\[
\Rcal(\ww{s}\ast^\bullet\ww{n})
=
\Rcal(\ww{s})\ast^\bullet\Rcal(\ww{n}),
\]
which completes the proof.
\end{proof}

\begin{lem}\label{lem:degree-convolution}
For every $s\in\NN$ and $d\geq0$, fix $\mathscr{S}_d(s)\in B$, and fix $D\in\NN\cup\{\infty\}$. When $D=\infty$, assume that
$B$ is a Hausdorff topological $R$-algebra and that the series below converge in $B$. Define the $R$-linear maps $\ev,\ev^\ant$ by
$\ev(1)=\ev^\ant(1)=1$ and, for every non-empty
$\ww{s}=(s_1,\ldots,s_r)$, by
\begin{align*}
\ev([\ww{s}])
&=\sum_{D>d_1>\cdots>d_r\geq0}
\prod_{k=1}^r\mathscr{S}_{d_k}(s_k),\\
\ev^\ant([\ww{s}])
&=(-1)^r\sum_{0\leq d_1\leq\cdots\leq d_r<D}
\prod_{k=1}^r\mathscr{S}_{d_k}(s_k).
\end{align*}
Then Conditions \ref{cond:unital} and \ref{cond:convolution} hold. 
\end{lem}
\begin{proof}
It suffices to prove the assertion for $D\in\NN$,
since the case $D=\infty$ follows by taking the limit
as $D\to\infty$. Condition \ref{cond:unital} follows immediately from
$\ev(1)=\ev^\ant(1)=1$. Let $\ww{s}=(s_1,\ldots,s_r)\in\Ical$ be non-empty. We first prove
\[
\sum_{i=0}^r
\ev(\ww{s}[:i])\ev^\ant(\ww{s}[i+1:])
=0.
\]
Set $H_0=H_{r+1}=0$,
and, for $1\leq i\leq r$, define
\[
H_i
=
(-1)^{r-i}
\sum_{\substack{
D>d_1>\cdots>d_i\geq0\\
d_i\leq d_{i+1}\leq\cdots\leq d_r<D}}
\prod_{k=1}^r\mathscr{S}_{d_k}(s_k).
\]
In particular,
\[
H_1=-\ev^\ant(\ww{s}),
\qquad
H_r=\ev(\ww{s}).
\]
We claim that, for every $0\leq i\leq r$,
\begin{equation}\label{eq:degree-telescoping}
\ev(\ww{s}[:i])\ev^\ant(\ww{s}[i+1:])
=
H_i-H_{i+1}.
\end{equation}
For $i=0$, this follows from
\[
H_0-H_1
=
\ev^\ant(\ww{s})
=
\ev(\varnothing)\ev^\ant(\ww{s}),
\]
while for $i=r$,
\[
H_r-H_{r+1}
=
\ev(\ww{s})
=
\ev(\ww{s})\ev^\ant(\varnothing).
\]
For $1\leq i\leq r-1$, we have
\[
\begin{aligned}
&\ev(\ww{s}[:i])\ev^\ant(\ww{s}[i+1:])=
(-1)^{r-i}
\sum_{\substack{
D>d_1>\cdots>d_i\geq0\\
0\leq d_{i+1}\leq\cdots\leq d_r<D}}
\prod_{k=1}^r\mathscr{S}_{d_k}(s_k).
\end{aligned}
\]
There is no condition relating $d_i$ and $d_{i+1}$ in this sum.
Splitting it into the two disjoint cases
\[
d_i\leq d_{i+1}
\qquad\text{and}\qquad
d_i>d_{i+1},
\]
gives
\[
\ev(\ww{s}[:i])\ev^\ant(\ww{s}[i+1:])
=
H_i-H_{i+1}.
\]
Thus \eqref{eq:degree-telescoping} holds for every $0\leq i\leq r$.
Summing over $i$ yields
\[
\begin{aligned}
\sum_{i=0}^r
\ev(\ww{s}[:i])\ev^\ant(\ww{s}[i+1:])
=
\sum_{i=0}^r(H_i-H_{i+1})=
H_0-H_{r+1}
=0.
\end{aligned}
\]
Similarly,
\[
\sum_{i=0}^r
\ev^{\ant}(\ww{s}[:i])\ev(\ww{s}[i+1:])
=0.
\]
Thus, Condition \ref{cond:convolution} holds.
\end{proof}

\begin{lem}\label{lem:image-comparison}
Let $\bullet\in\{\zeta,\Li\}$. Assume Conditions \ref{cond:unital} and
\ref{cond:convolution}. Then the following hold:
\begin{enumerate}[label=\textup{(\roman*)}]
\item If Condition \ref{cond:ev-product} holds, then
\[
\ev^\ant(\Hfk_R)\subseteq\ev(\Hfk_R).
\]

\item If Condition \ref{cond:ant-product} holds, then
\[
\ev(\Hfk_R)\subseteq\ev^\ant(\Hfk_R).
\]
\end{enumerate}
In particular, if both Conditions \ref{cond:ev-product}
and \ref{cond:ant-product} hold, then
\[
\ev(\Hfk_R)=\ev^\ant(\Hfk_R).
\]
\end{lem}

\begin{proof}
Suppose that Condition \ref{cond:ev-product} holds. We prove $\ev^{\ant}(\ww{s})\in\ev(\Hfk_R)$ by induction on $\dep(\ww{s})$. The case $\ww{s}=\varnothing$ follows from Condition \ref{cond:unital}. For any non-empty index
$\ww{s}=(s_1,\ldots,s_r)\in\Ical$, Condition 
\ref{cond:convolution} gives
\[
\ev^\ant(\ww{s})
=-\sum_{i=1}^r
\ev(\ww{s}[:i])\ev^\ant(\ww{s}[i+1:]).
\]
Since $\ev(\Hfk_R)$ is closed under multiplication by
Condition \ref{cond:ev-product}, the induction hypothesis
implies that $\ev^\ant([\ww{s}])\in\ev(\Hfk_R)$.
By $R$-linearity, this proves
$\ev^\ant(\Hfk_R)\subseteq\ev(\Hfk_R)$.
On the other hand, if Condition \ref{cond:ant-product} holds, the same argument shows that
$\ev(\Hfk_R)\subseteq\ev^\ant(\Hfk_R)$. This completes the proof.
\end{proof}

\subsection{Deconcatenation of the Linear Relations}\label{sec:deconcatenation-relations}

We next formulate a consequence of the preceding deconcatenation
formalism which will be used for both realizations.  Let $J\subseteq B$
be an ideal and write
\[
\pi:B\longrightarrow \ov B
\]
for the quotient map where $\ov{B}=B/J$.  Put
\[
\ov{\ev}=\pi\circ\ev,
\qquad
\ov{\ev}^{\ant}=\pi\circ\ev^\ant.
\]
Let $\bullet\in\{\zeta,\Li\}$.  We assume Conditions \ref{cond:unital}, \ref{cond:convolution}, and
\ref{cond:ev-product}, and in addition assume that
\begin{equation}\label{eq:quotient-vanishing}
\begin{cases}
\ov{\ev}([j])=0 & \text{whenever $q-1\mid j$, if $\bullet=\zeta$},\\
\ov{\ev}([q-1])=0 & \text{if $\bullet=\Li$}.
\end{cases}
\end{equation}

The maps $\ov{\ev}$ and $\ov{\ev}^{\ant}$ satisfy Conditions
\ref{cond:unital} and \ref{cond:convolution}.  Moreover, for every
$j\in\NN$, Condition \ref{cond:convolution} in depth one gives
\[
\ov{\ev}^{\ant}([j])=-\ov{\ev}([j]).
\]
Hence \eqref{eq:quotient-vanishing} implies Condition
\ref{cond:vanishing} for the pair
$(\ov{\ev},\ov{\ev}^{\ant})$.  Applying
\cref{lem:bullet-stuffle-transfer} to this pair shows that
$\ov{\ev}^{\ant}$ is an $R$-algebra homomorphism with respect to $\ast^\bullet$.
We may therefore apply \cref{lem:bullet-stuffle-transfer} again to the pair
$(\ov{\ev}^{\ant},\ov{\ev})$.  If we put
\[
\Rcal_\pi
=
(\id\otimes\ov{\ev})\Delta:
\Hfk_R\longrightarrow\Hfk_R\otimes_R\ov B
\]
and equip $\Hfk_R\otimes_R\ov B$ with the product
\[
(P\otimes x)\ast^\bullet(Q\otimes y)
=
(P\ast^\bullet Q)\otimes xy,
\]
then we obtain
\begin{equation}\label{eq:quotient-product-transfer}
\Rcal_\pi(P\ast^\bullet Q)
=
\Rcal_\pi(P)\ast^\bullet\Rcal_\pi(Q)
\qquad(P,Q\in\Hfk_R).
\end{equation}

For the remainder of this subsection, we abbreviate $\alpha=\alpha_{1;(q-1)}^\bullet$ (recall \cref{df:alpha}).
For every non-empty index
$\ww{s}=(s_1,\ldots,s_r)\in\Ical$, we define the $R$-linear map
\[
B_{\ww{s}}^\bullet:\Hfk_R\longrightarrow\Hfk_R
\]
by
\begin{equation}\label{eq:B-bullet-definition}
B_{\ww{s}}^\bullet(P)
=
[\ww{s},P]
+
\ww{s}\boxplus P
+
\varepsilon_\bullet
[\ww{s}^{+},D_{s_r}(P)]\qquad(P\in\Hfk_R).
\end{equation}
Thus $B_{\ww{s}}^\bullet(1)=[\ww{s}]$.  For $m\geq1$ and
$\ww{n}\in\Ical$, put
\begin{equation}\label{eq:P-bullet-definition}
P_m^\bullet(\ww{n})
=
[\{q\}^m,\ww{n}]
+
\{q\}^m\boxplus\ww{n}
+
\varepsilon_\bullet
[\{q\}^{m-1},D_q(\ww{n})].
\end{equation}
Equivalently,
\[
P_m^\bullet(\ww{n})
=
B_{\{q\}^m}^\bullet(\ww{n}).
\]
With this notation, we have
(recall \cref{df:Mishiba-relations})
\begin{equation}\label{eq:coaction-A-PB}
\sA^\bullet(\ww{s};m;\ww{n})
=
[\ww{s},P_m^\bullet(\ww{n})]
-
L_1^m
B_{\ww{s}}^\bullet\left(\alpha^m(\ww{n})\right)
\end{equation}
for non-empty $\ww{s}\in\Ical$ and
\begin{equation}\label{eq:coaction-A-empty}
\sA^\bullet(\varnothing;m;\ww{n})
=
P_m^\bullet(\ww{n})
-
L_1^m\alpha^m(\ww{n}).
\end{equation}

We first establish two auxiliary formulas.

\begin{lem}\label{lem:formal-coaction-cuts}
Under the assumptions above, let $\ww{n}\in\Ical$ with $\dep(\ww{n})=\ell$.  Then, for every
$m\geq0$,
\begin{equation}\label{eq:coaction-alpha-cut}
\begin{aligned}
\Rcal_\pi\left(\alpha^m(\ww{n})\right)
={}&
1\otimes\ov{\ev}\left(\alpha^m(\ww{n})\right)+
\sum_{i=1}^{m}
\alpha^i(1)\otimes
\ov{\ev}\left(\alpha^{m-i}(\ww{n})\right)\\
&+
\sum_{k=1}^{\ell}
\alpha^m(\ww{n}[:k])\otimes
\ov{\ev}(\ww{n}[k+1:]).
\end{aligned}
\end{equation}
Moreover, for every $c\in\NN$,
\begin{equation}\label{eq:coaction-D-cut}
\begin{aligned}
\varepsilon_\bullet\Rcal_\pi(D_c(\ww{n}))
={}&
\varepsilon_\bullet
\left(
1\otimes\ov{\ev}(D_c(\ww{n}))
+
\sum_{k=1}^{\ell}
D_c(\ww{n}[:k])\otimes
\ov{\ev}(\ww{n}[k+1:])
\right).
\end{aligned}
\end{equation}
\end{lem}

\begin{proof}
By \eqref{eq:quotient-vanishing} and \eqref{eq:quotient-product-transfer}, we have
\[
\Rcal_\pi([q-1])=[q-1]\otimes1,
\qquad
\Rcal_\pi([q-1]\ast^\bullet P)
=
([q-1]\otimes1)\ast^\bullet\Rcal_\pi(P)
\]
for every $P\in\Hfk_R$. Applying
\eqref{eq:R-left-letter} to the pair
$(\ov{\ev}^{\ant},\ov{\ev})$ yields
\begin{align}\label{eq:alpha-R-recursion}
\Rcal_\pi(\alpha(P))
&=
\Rcal_\pi\left([1,[q-1]\ast^\bullet P]\right)\notag\\
&=
1\otimes\ov{\ev}(\alpha(P))
+
([1,-]\otimes\id)
\Rcal_\pi([q-1]\ast^\bullet P)\notag\\
&=
1\otimes\ov{\ev}(\alpha(P))
+
([1,-]\otimes\id)
\left(
([q-1]\otimes1)\ast^\bullet\Rcal_\pi(P)
\right)\notag\\
&=
1\otimes\ov{\ev}(\alpha(P))
+
(\alpha\otimes\id)\Rcal_\pi(P).
\end{align}
We prove \eqref{eq:coaction-alpha-cut} by induction on $m$. For $m=0$, the result is clear. Now suppose that the result holds for some $m\geq0$. Then
by \eqref{eq:alpha-R-recursion},
\begin{align*}
\Rcal_\pi(\alpha^{m+1}(\ww{n}))
&=
1\otimes
\ov{\ev}(\alpha^{m+1}(\ww{n}))
+
(\alpha\otimes\id)
\Rcal_\pi(\alpha^m(\ww{n}))\\
&=
1\otimes
\ov{\ev}(\alpha^{m+1}(\ww{n}))+
\alpha(1)\otimes
\ov{\ev}(\alpha^m(\ww{n}))\\
&\quad+
\sum_{i=1}^m
\alpha^{i+1}(1)\otimes
\ov{\ev}(\alpha^{m-i}(\ww{n}))+
\sum_{j=1}^\ell
\alpha^{m+1}(\ww{n}[:j])\otimes
\ov{\ev}(\ww{n}[j+1:])\\
&=
1\otimes
\ov{\ev}(\alpha^{m+1}(\ww{n}))+
\sum_{i=1}^{m+1}
\alpha^i(1)\otimes
\ov{\ev}(\alpha^{m+1-i}(\ww{n}))\\
&\quad+
\sum_{k=1}^\ell
\alpha^{m+1}(\ww{n}[:k])\otimes
\ov{\ev}(\ww{n}[k+1:]).
\end{align*}

Next, we prove \eqref{eq:coaction-D-cut}.
If $\bullet=\Li$, then \eqref{eq:coaction-D-cut} is immediate from
$\varepsilon_\Li=0$.  Suppose that $\bullet=\zeta$.  The assertion is
clear for $\ww{n}=\varnothing$.  If
$\ww{n}=(n_1,\ldots,n_\ell)\in\Ical$ is non-empty, then
\eqref{eq:D-zeta-definition} gives
\[
D_c(\ww{n})
=
\sum_{j=1}^{c+n_1-1}
\Delta_{c,n_1}^{[j]}
[c+n_1-j,[j]\ast^\zeta\ww{n}^-].
\]
Whenever $\Delta_{c,n_1}^{[j]}\neq0$, we have $q-1\mid j$, so
\eqref{eq:quotient-vanishing} implies
\[
\Rcal_\pi([j])=[j]\otimes1.
\]
Using \eqref{eq:quotient-product-transfer}, we obtain
\[
\Rcal_\pi([j]\ast^\zeta\ww{n}^-)
=
\sum_{k=1}^{\ell}
([j]\ast^\zeta\ww{n}[:k]^-)
\otimes\ov{\ev}(\ww{n}[k+1:]).
\]
Applying \eqref{eq:R-left-letter} to each summand and summing over $j$
proves \eqref{eq:coaction-D-cut}.
\end{proof}

We now derive the corresponding formulas for
$B_{\ww{s}}^\bullet$ and $P_m^\bullet$.

\begin{prop}
Under the assumptions above, for every non-empty
$\ww{s}=(s_1,\ldots,s_r)\in\Ical$ and every $P\in\Hfk_R$, we have
\begin{equation}\label{eq:coaction-B-cut}
\begin{aligned}
\Rcal_\pi(B_{\ww{s}}^\bullet(P))
={}
1\otimes\ov{\ev}(B_{\ww{s}}^\bullet(P))+
\sum_{j=1}^{r-1}
[\ww{s}[:j]]\otimes
\ov{\ev}(B_{\ww{s}[j+1:]}^\bullet(P))
+
(B_{\ww{s}}^\bullet\otimes\id)\Rcal_\pi(P).
\end{aligned}
\end{equation}
In particular, for every $m\geq1$ and $\ww{n}\in\Ical$ with $\dep(\ww{n})=\ell$, we have
\begin{equation}\label{eq:coaction-P-cut}
\begin{aligned}
\Rcal_\pi(P_m^\bullet(\ww{n}))
={}&
1\otimes\ov{\ev}(P_m^\bullet(\ww{n}))+
\sum_{i=1}^{m-1}
[\{q\}^i]\otimes
\ov{\ev}(P_{m-i}^\bullet(\ww{n}))\\
&+
[\{q\}^m]\otimes\ov{\ev}(\ww{n})+
\sum_{k=1}^{\ell}
P_m^\bullet(\ww{n}[:k])\otimes
\ov{\ev}(\ww{n}[k+1:]).
\end{aligned}
\end{equation}
\end{prop}
\begin{proof}
By $R$-linearity, it suffices to take
$P=[\ww{n}]$, where $\ww{n}\in\Ical$ and
$\dep(\ww{n})=\ell$.

When $\bullet=\Li$, we have $\varepsilon_\bullet=0$, and
\eqref{eq:coaction-B-cut} follows directly from the deconcatenation of $[\ww{s},\ww{n}]+\ww{s}\boxplus\ww{n}$.
Now suppose that $\bullet=\zeta$. By
\eqref{eq:coaction-D-cut}, we have
\[
\Rcal_\pi(D_{s_r}(\ww{n}))
=
1\otimes\ov{\ev}(D_{s_r}(\ww{n}))
+
\sum_{k=1}^{\ell}
D_{s_r}(\ww{n}[:k])
\otimes
\ov{\ev}(\ww{n}[k+1:]).
\]
Hence, by \cref{lem:R-concatenation},
\[
\begin{aligned}
\Rcal_\pi
\left(
[\ww{s}^{+},D_{s_r}(\ww{n})]
\right)
={}&
1\otimes
\ov{\ev}
\left(
[\ww{s}^{+},D_{s_r}(\ww{n})]
\right)\\
&+
\sum_{j=1}^{r-1}
[\ww{s}[:j]]
\otimes
\ov{\ev}
\left(
[\ww{s}[j+1:]^{+},D_{s_r}(\ww{n})]
\right)\\
&+
\sum_{k=1}^{\ell}
[\ww{s}^{+},D_{s_r}(\ww{n}[:k])]
\otimes
\ov{\ev}(\ww{n}[k+1:]).
\end{aligned}
\]
Combining this with the deconcatenation of
$[\ww{s},\ww{n}]+\ww{s}\boxplus\ww{n}$, we obtain
\begin{align}\label{eq:Rpi(B)}
\Rcal_\pi(B_{\ww{s}}^\bullet(\ww{n}))
={}&
1\otimes
\ov{\ev}(B_{\ww{s}}^\bullet(\ww{n}))+
\sum_{j=1}^{r-1}
[\ww{s}[:j]]
\otimes
\ov{\ev}
\left(
B_{\ww{s}[j+1:]}^\bullet(\ww{n})
\right)\\\notag
&+
[\ww{s}]\otimes\ov{\ev}(\ww{n})+
\sum_{k=1}^{\ell}
B_{\ww{s}}^\bullet(\ww{n}[:k])
\otimes
\ov{\ev}(\ww{n}[k+1:]).
\end{align}
Since
\[
\Rcal_\pi(\ww{n})
=
1\otimes\ov{\ev}(\ww{n})
+
\sum_{k=1}^{\ell}
\ww{n}[:k]\otimes
\ov{\ev}(\ww{n}[k+1:])
\]
and $B_{\ww{s}}^\bullet(1)=[\ww{s}]$, the last two terms are equal to
$(B_{\ww{s}}^\bullet\otimes\id)
\Rcal_\pi(\ww{n})$.
This proves \eqref{eq:coaction-B-cut}.

Finally, let $m\geq1$. Since
$P_m^\bullet(\ww{n})
=
B_{\{q\}^m}^\bullet(\ww{n})$,
applying \eqref{eq:Rpi(B)} with
$\ww{s}=\{q\}^m$ gives
\[
\begin{aligned}
\Rcal_\pi(P_m^\bullet(\ww{n}))
={}&
1\otimes\ov{\ev}(P_m^\bullet(\ww{n}))
+
\sum_{i=1}^{m-1}
[\{q\}^i]\otimes
\ov{\ev}(P_{m-i}^\bullet(\ww{n}))\\
&+
[\{q\}^m]\otimes\ov{\ev}(\ww{n})+
\sum_{k=1}^{\ell}
P_m^\bullet(\ww{n}[:k])
\otimes
\ov{\ev}(\ww{n}[k+1:]),
\end{aligned}
\]
which completes the proof.
\end{proof}

\subsection{The Formal Coaction}\label{sec:formal-coaction}
We are now ready to construct the formal coaction.

\begin{prop}\label{prop:formal-coaction}
Let $\bullet\in\{\zeta,\Li\}$.  Assume that the pair $(\ev,\ev^\ant)$ satisfies Conditions
\ref{cond:unital}, \ref{cond:convolution}, and
\ref{cond:ev-product}, and that \eqref{eq:quotient-vanishing} holds.
Suppose moreover that
\[
\ker(\ev)=\mathscr{R}_R^\bullet=\Span_R\left\{
\sA^\bullet(\ww{s};m;\ww{n})
:
\ww{s},\ww{n}\in\Ical,\ m\geq1
\right\}.
\]
Then the assignment
\begin{equation}\label{eq:formal-coaction-map}
\begin{aligned}
\widetilde{\Delta}_{\ev}:
\ev(\Hfk_R)
&\longrightarrow
\ev(\Hfk_R)\otimes_R\ov{\ev}(\Hfk_R),\\
\ev(P)
&\longmapsto
(\ev\otimes\ov{\ev})\Delta(P)
\end{aligned}
\end{equation}
is a well-defined $R$-algebra homomorphism.  In particular, for every
index $\ww{s}\in\Ical$ with $\dep(\ww{s})=r$,
\[
\widetilde{\Delta}_{\ev}(\ev([\ww{s}]))
=
\sum_{i=0}^r
\ev([\ww{s}[:i]])\otimes
\ov{\ev}([\ww{s}[i+1:]]).
\]
\end{prop}
\begin{proof}
Note that we have
\[
(\ev\otimes\ov{\ev})\Delta(P)=(\ev\otimes\id)\Rcal_\pi(P)
\]
for any $P\in\Hfk_R$ and 
\begin{equation}\label{eq:bar-ev-A-zero}
\ov{\ev}
\left(
\sA^\bullet(\ww{s};m;\ww{n})
\right)
=0
\end{equation}
for all $\ww{s},\ww{n}\in\Ical$ and $m\geq1$. We first prove the well-definedness. It suffices to check that
\begin{equation}\label{eq:formal-coaction-kernel}
(\ev\otimes\id)\Rcal_\pi
\left(
\sA^\bullet(\ww{s};m;\ww{n})
\right)
=0
\end{equation}
for all $\ww{s},\ww{n}\in\Ical$ and $m\geq1$.

Suppose first that
$\ww{s}=(s_1,\ldots,s_r)\in\Ical$ is non-empty, and write
$\dep(\ww{n})=\ell$. By \eqref{eq:coaction-A-PB},
\[
\sA^\bullet(\ww{s};m;\ww{n})
=
[\ww{s},P_m^\bullet(\ww{n})]
-
L_1^m
B_{\ww{s}}^\bullet
\left(
\alpha^m(\ww{n})
\right).
\]
We compute the two terms separately. By \cref{lem:R-concatenation} and
\eqref{eq:coaction-P-cut},
\[
\begin{aligned}
\Rcal_\pi
\left(
[\ww{s},P_m^\bullet(\ww{n})]
\right)
={}&
1\otimes
\ov{\ev}
\left(
[\ww{s},P_m^\bullet(\ww{n})]
\right)\\
&+
\sum_{j=1}^{r-1}
[\ww{s}[:j]]
\otimes
\ov{\ev}
\left(
[\ww{s}[j+1:],P_m^\bullet(\ww{n})]
\right)\\
&+
[\ww{s}]
\otimes
\ov{\ev}(P_m^\bullet(\ww{n}))\\
&+
\sum_{i=1}^{m-1}
[\ww{s},\{q\}^i]
\otimes
\ov{\ev}(P_{m-i}^\bullet(\ww{n}))\\
&+
[\ww{s},\{q\}^m]
\otimes
\ov{\ev}(\ww{n})\\
&+
\sum_{k=1}^{\ell}
[\ww{s},P_m^\bullet(\ww{n}[:k])]
\otimes
\ov{\ev}(\ww{n}[k+1:]).
\end{aligned}
\]

On the other hand, by
\eqref{eq:coaction-B-cut} and
\eqref{eq:coaction-alpha-cut},
\[
\begin{aligned}
\Rcal_\pi
\left(
B_{\ww{s}}^\bullet(\alpha^m(\ww{n}))
\right)
={}&
1\otimes
\ov{\ev}
\left(
B_{\ww{s}}^\bullet(\alpha^m(\ww{n}))
\right)\\
&+
\sum_{j=1}^{r-1}
[\ww{s}[:j]]
\otimes
\ov{\ev}
\left(
B_{\ww{s}[j+1:]}^\bullet
(\alpha^m(\ww{n}))
\right)\\
&+
[\ww{s}]
\otimes
\ov{\ev}(\alpha^m(\ww{n}))\\
&+
\sum_{i=1}^{m}
B_{\ww{s}}^\bullet(\alpha^i(1))
\otimes
\ov{\ev}(\alpha^{m-i}(\ww{n}))\\
&+
\sum_{k=1}^{\ell}
B_{\ww{s}}^\bullet
(\alpha^m(\ww{n}[:k]))
\otimes
\ov{\ev}(\ww{n}[k+1:]),
\end{aligned}
\]
where we have used
$B_{\ww{s}}^\bullet(1)=[\ww{s}]$.

Subtracting $L_1^m$ times the second identity from the first and
applying $\ev\otimes\id$, we obtain
\begin{align}
&(\ev\otimes\id)\Rcal_\pi
\left(
\sA^\bullet(\ww{s};m;\ww{n})
\right)
\notag\\
={}&
1\otimes
\ov{\ev}
\left(
\sA^\bullet(\ww{s};m;\ww{n})
\right)
\notag\\
&+
\sum_{j=1}^{r-1}
\ev(\ww{s}[:j])
\otimes
\ov{\ev}
\left(
\sA^\bullet(\ww{s}[j+1:];m;\ww{n})
\right)
\notag\\
&+
\ev(\ww{s})
\otimes
\ov{\ev}
\left(
\sA^\bullet(\varnothing;m;\ww{n})
\right)
\notag\\
&+
\sum_{i=1}^{m-1}
\left(
\ev([\ww{s},\{q\}^i])
\otimes
\ov{\ev}(P_{m-i}^\bullet(\ww{n}))
\notag-L_1^m
\ev(B_{\ww{s}}^\bullet(\alpha^i(1)))
\otimes
\ov{\ev}(\alpha^{m-i}(\ww{n}))
\right)
\notag\\
&+
\left(
\ev([\ww{s},\{q\}^m])
-
L_1^m
\ev(B_{\ww{s}}^\bullet(\alpha^m(1)))
\right)
\otimes
\ov{\ev}(\ww{n})
\notag\\
&+
\sum_{k=1}^{\ell}
\ev
\left(
\sA^\bullet(\ww{s};m;\ww{n}[:k])
\right)
\otimes
\ov{\ev}(\ww{n}[k+1:]).
\label{eq:formal-coaction-expansion}
\end{align}
The first three terms in
\eqref{eq:formal-coaction-expansion} vanish by
\eqref{eq:bar-ev-A-zero}. The last sum also vanishes, since
\[
\sA^\bullet(\ww{s};m;\ww{n}[:k])
\in\mathscr{R}_R^\bullet=\ker(\ev).
\]
It remains to consider the terms involving the index $i$. For every
$i\geq1$, we have
\[
\sA^\bullet(\ww{s};i;\varnothing)
\in\mathscr{R}_R^\bullet=\ker(\ev).
\]
Since
$P_i^\bullet(\varnothing)=[\{q\}^i]$,
\eqref{eq:coaction-A-PB} implies that
\begin{equation}\label{eq:formal-empty-right}
\ev([\ww{s},\{q\}^i])
=
L_1^i
\ev
\left(
B_{\ww{s}}^\bullet(\alpha^i(1))
\right).
\end{equation}
For $i=m$, this shows that the fifth term in
\eqref{eq:formal-coaction-expansion} is zero. For
$1\leq i<m$, using \eqref{eq:formal-empty-right}, the corresponding
summand becomes
\[
\begin{aligned}
&\ev([\ww{s},\{q\}^i])
\otimes
\ov{\ev}(P_{m-i}^\bullet(\ww{n}))
-
L_1^m
\ev(B_{\ww{s}}^\bullet(\alpha^i(1)))
\otimes
\ov{\ev}(\alpha^{m-i}(\ww{n}))\\
={}&
\ev([\ww{s},\{q\}^i])
\otimes
\left(
\ov{\ev}(P_{m-i}^\bullet(\ww{n}))
-
L_1^{m-i}
\ov{\ev}(\alpha^{m-i}(\ww{n}))
\right)\\
={}&
\ev([\ww{s},\{q\}^i])
\otimes
\ov{\ev}
\left(
\sA^\bullet(\varnothing;m-i;\ww{n})
\right)
=0,
\end{aligned}
\]
where the last equality follows from
\eqref{eq:bar-ev-A-zero}.
Therefore \eqref{eq:formal-coaction-kernel} holds whenever
$\ww{s}\in\Ical$ is non-empty.

We now consider the case $\ww{s}=\varnothing$. By
\eqref{eq:coaction-A-empty},
\[
\sA^\bullet(\varnothing;m;\ww{n})
=
P_m^\bullet(\ww{n})
-
L_1^m\alpha^m(\ww{n}).
\]
Using
\eqref{eq:coaction-alpha-cut} and \eqref{eq:coaction-P-cut}, and then applying
$\ev\otimes\id$, we obtain
\[
\begin{aligned}
&(\ev\otimes\id)\Rcal_\pi
\left(
\sA^\bullet(\varnothing;m;\ww{n})
\right)\\
={}&
1\otimes
\ov{\ev}
\left(
\sA^\bullet(\varnothing;m;\ww{n})
\right)\\
&+
\sum_{i=1}^{m-1}
\left(
\ev([\{q\}^i])
\otimes
\ov{\ev}(P_{m-i}^\bullet(\ww{n}))
-
L_1^m
\ev(\alpha^i(1))
\otimes
\ov{\ev}(\alpha^{m-i}(\ww{n}))
\right)\\
&+
\left(
\ev([\{q\}^m])
-
L_1^m\ev(\alpha^m(1))
\right)
\otimes
\ov{\ev}(\ww{n})\\
&+
\sum_{k=1}^{\ell}
\ev
\left(
\sA^\bullet(\varnothing;m;\ww{n}[:k])
\right)
\otimes
\ov{\ev}(\ww{n}[k+1:]).
\end{aligned}
\]
The first term is zero by \eqref{eq:bar-ev-A-zero}, while the last
sum vanishes since
\[
\sA^\bullet(\varnothing;m;\ww{n}[:k])
\in\ker(\ev).
\]
It remains to treat the terms involving $i$. For every $i\geq1$, we have
\[
\sA^\bullet(\varnothing;i;\varnothing)
\in\mathscr{R}_R^\bullet=\ker(\ev).
\]
Since
\[
\sA^\bullet(\varnothing;i;\varnothing)
=
[\{q\}^i]-L_1^i\alpha^i(1),
\]
it follows that
\[
\ev([\{q\}^i])
=
L_1^i\ev(\alpha^i(1)).
\]
For $i=m$, this shows that the third term in the above expansion is
zero. For $1\leq i<m$, the corresponding summand becomes
\[
\begin{aligned}
&\ev([\{q\}^i])
\otimes
\ov{\ev}(P_{m-i}^\bullet(\ww{n}))
-
L_1^m
\ev(\alpha^i(1))
\otimes
\ov{\ev}(\alpha^{m-i}(\ww{n}))\\
={}&
\ev([\{q\}^i])
\otimes
\ov{\ev}
\left(
\sA^\bullet(\varnothing;m-i;\ww{n})
\right)
=0.
\end{aligned}
\]
Therefore \eqref{eq:formal-coaction-kernel} also holds for
$\ww{s}=\varnothing$. Hence
\eqref{eq:formal-coaction-kernel} holds for all
$\ww{s},\ww{n}\in\Ical$ and $m\geq1$, and the assignment
\eqref{eq:formal-coaction-map} is well-defined.

Finally, by Condition \ref{cond:ev-product} and
\eqref{eq:quotient-product-transfer}, for any
$P,Q\in\Hfk_R$, we have
\[
\begin{aligned}
\widetilde{\Delta}_{\ev}
\left(
\ev(P)\ev(Q)
\right)
&=
(\ev\otimes\id)
\Rcal_\pi(P\ast^\bullet Q)\\
&=
(\ev\otimes\id)
\left(
\Rcal_\pi(P)\ast^\bullet\Rcal_\pi(Q)
\right)\\
&=
\widetilde{\Delta}_{\ev}(\ev(P))
\widetilde{\Delta}_{\ev}(\ev(Q)).
\end{aligned}
\]
Moreover, $\widetilde{\Delta}_{\ev}(1)=1\otimes1$.
Thus, $\widetilde{\Delta}_{\ev}$ is an $R$-algebra homomorphism.
\end{proof}

We next give the corresponding formula for the $\ast$-inverse realization.

\begin{prop}\label{prop:formal-coaction-ant}
Under the assumptions of \cref{prop:formal-coaction}, for every
index $\ww{s}\in\Ical$ with $\dep(\ww{s})=r$, we have
\[
\widetilde{\Delta}_{\ev}
\left(
\ev^\ant([\ww{s}])
\right)
=
\sum_{i=0}^r
\ev^\ant([\ww{s}[i+1:]])
\otimes
\ov{\ev}^{\ant}([\ww{s}[:i]]).
\]
\end{prop}

\begin{proof}
By \cref{lem:image-comparison}, we have
\[
\ev^\ant(\Hfk_R)\subseteq\ev(\Hfk_R),
\]
so the left-hand side is well-defined. We prove the formula by
induction on $r$. The case $r=0$ is immediate. Let $r\geq1$. By Condition \ref{cond:convolution},
\begin{equation}\label{eq:ant-convolution-recursion}
\ev^\ant([\ww{s}])
=
-\sum_{i=1}^r
\ev([\ww{s}[:i]])
\ev^\ant([\ww{s}[i+1:]]).
\end{equation}
Since $\widetilde{\Delta}_{\ev}$ is an $R$-algebra homomorphism,
\cref{prop:formal-coaction} and the induction hypothesis imply that
\[
\begin{aligned}
\widetilde{\Delta}_{\ev}
\left(
\ev^\ant([\ww{s}])
\right)
={}&
-\sum_{i=1}^r
\left(
\sum_{j=0}^i
\ev([\ww{s}[:j]])
\otimes
\ov{\ev}([s_{j+1},\ldots,s_i])
\right)\\
&\qquad\qquad\cdot
\left(
\sum_{k=i}^r
\ev^\ant([\ww{s}[k+1:]])
\otimes
\ov{\ev}^{\ant}([s_{i+1},\ldots,s_k])
\right).
\end{aligned}
\]
Expanding the
product gives
\[
\begin{aligned}
&\widetilde{\Delta}_{\ev}
\left(
\ev^\ant([\ww{s}])
\right)
={}
-\sum_{\substack{1\leq i\leq k\leq r}}
\ev^\ant([\ww{s}[k+1:]])
\otimes
\ov{\ev}([\ww{s}[:i]])
\ov{\ev}^{\ant}([s_{i+1},\ldots,s_k])\\
&-
\sum_{\substack{1\leq j\leq i\leq k\leq r}}
\ev([\ww{s}[:j]])
\ev^\ant([\ww{s}[k+1:]])
\otimes
\ov{\ev}([s_{j+1},\ldots,s_i])
\ov{\ev}^{\ant}([s_{i+1},\ldots,s_k]).
\end{aligned}
\]
The first sum consists of the terms with $j=0$, while the second
contains those with $j\geq1$.

For the first sum, we sum first over $k$. It becomes
\[
\begin{aligned}
-\sum_{k=1}^r
\ev^\ant([\ww{s}[k+1:]])
\otimes
\Bigg(
&\sum_{i=1}^{k-1}
\ov{\ev}([\ww{s}[:i]])
\ov{\ev}^{\ant}([s_{i+1},\ldots,s_k])+
\ov{\ev}([\ww{s}[:k]])
\Bigg).
\end{aligned}
\]
By Condition \ref{cond:convolution} applied to the index
$\ww{s}[:k]$, we have
\[
\ov{\ev}^{\ant}([\ww{s}[:k]])
+
\sum_{i=1}^{k-1}
\ov{\ev}([\ww{s}[:i]])
\ov{\ev}^{\ant}([s_{i+1},\ldots,s_k])
+
\ov{\ev}([\ww{s}[:k]])
=0.
\]
Hence the first sum is equal to
\[
\sum_{k=1}^r
\ev^\ant([\ww{s}[k+1:]])
\otimes
\ov{\ev}^{\ant}([\ww{s}[:k]]).
\]

We next consider the second sum. Separating the terms with $j=k$,
and then summing first over $i$ in the remaining terms, we obtain
\[
\begin{aligned}
&-
\sum_{1\leq j<k\leq r}
\ev([\ww{s}[:j]])
\ev^\ant([\ww{s}[k+1:]])
\\&\otimes
\left(
\ov{\ev}^{\ant}([s_{j+1},\ldots,s_k])
+
\sum_{i=j+1}^{k-1}
\ov{\ev}([s_{j+1},\ldots,s_i])
\ov{\ev}^{\ant}([s_{i+1},\ldots,s_k])
+
\ov{\ev}([s_{j+1},\ldots,s_k])
\right)\\
&\qquad
-
\sum_{j=1}^r
\ev([\ww{s}[:j]])
\ev^\ant([\ww{s}[j+1:]])
\otimes1.
\end{aligned}
\]
For $j<k$, Condition \ref{cond:convolution} applied to the index
$(s_{j+1},\ldots,s_k)$
shows that 
\[
\ov{\ev}^{\ant}([s_{j+1},\ldots,s_k])
+
\sum_{i=j+1}^{k-1}
\ov{\ev}([s_{j+1},\ldots,s_i])
\ov{\ev}^{\ant}([s_{i+1},\ldots,s_k])
+
\ov{\ev}([s_{j+1},\ldots,s_k])=0.
\]
Therefore, the
second sum reduces to
\[
-\sum_{j=1}^r
\ev([\ww{s}[:j]])
\ev^\ant([\ww{s}[j+1:]])
\otimes1
=
\ev^\ant([\ww{s}])\otimes1,
\]
where the last equality follows from
\eqref{eq:ant-convolution-recursion}.

Combining the two sums, we obtain
\[
\begin{aligned}
\widetilde{\Delta}_{\ev}
\left(
\ev^\ant([\ww{s}])
\right)
={}&
\ev^\ant([\ww{s}])\otimes1+
\sum_{k=1}^r
\ev^\ant([\ww{s}[k+1:]])
\otimes
\ov{\ev}^{\ant}([\ww{s}[:k]])\\
={}&
\sum_{k=0}^r
\ev^\ant([\ww{s}[k+1:]])
\otimes
\ov{\ev}^{\ant}([\ww{s}[:k]]),
\end{aligned}
\]
as desired.
\end{proof}

\section{The \texorpdfstring{$\infty$}{infinite}-adic Coaction and the
\texorpdfstring{$\ast$}{star}-Inverse}
\label{sec:infty-coaction}

In this section, we apply the formalism developed in
\S\ref{sec:formalism} to the $\infty$-adic multiple zeta value
realization. Recall that
\[
\ov{\Zcal}_{\infty,R}
=
\Zcal_{\infty,R}/\zeta_A(q-1)\Zcal_{\infty,R},
\]
and 
\[
\pi_\infty:
\Zcal_{\infty,R}
\longrightarrow
\ov{\Zcal}_{\infty,R}
\]
denotes the quotient map.

\subsection{The
\texorpdfstring{$\infty$}{infinite}-adic Coaction}
\label{sec:construction-infty-coaction}

We first verify the vanishing condition required in
\S\ref{sec:formal-coaction}.

\begin{lem}\label{lem:Euler-Carlitz-L1}
Let $j\in\NN$ with $q-1\mid j$. Then there exists
$c_j\in\FF_p(L_1)^\times$ such that
\[
\zeta_A(j)
=
c_j\zeta_A(q-1)^{j/(q-1)}.
\]
In particular, we have
\[
\pi_\infty(\zeta_A(j))=0.
\]
\end{lem}

\begin{proof}
By the Euler--Carlitz formula (see, for example,
\cite[\S5.2]{Tha04}),
\[
\zeta_A(j)
=
\frac{\BC(j)}{\Gamma_{j+1}}
\widetilde{\pi}^{\,j},
\qquad
\zeta_A(q-1)
=
\frac{\widetilde{\pi}^{\,q-1}}{L_1},
\]
where $\BC(j)$ is the $j$-th Bernoulli--Carlitz number,
$\Gamma_{j+1}$ is the $(j+1)$-st Carlitz factorial and $\widetilde{\pi}$ is a fixed Carlitz period (see \cite{Tha04}). 
Writing $j=m(q-1)$, we obtain
\[
\zeta_A(j)
=
\frac{\BC(j)L_1^m}{\Gamma_{j+1}}
\zeta_A(q-1)^m.
\]
Thus it remains to check that
\[
\frac{\BC(j)L_1^m}{\Gamma_{j+1}}
\in
\FF_p(L_1)^\times.
\]

Recall that
\[
\frac{z}{\exp_C(z)}
=
\sum_{n\geq0}
\frac{\BC(n)}{\Gamma_{n+1}}z^n
\]
and
\[
\exp_C(z)
=
\sum_{i\geq0}\frac{z^{q^i}}{D_i},
\]
where
$D_0=1$ and
$D_i=(\theta^{q^i}-\theta)D_{i-1}^q$ for all $i\geq 1$.
Notice that we have
\[
\theta^{q^i}-\theta
=
\sum_{r=0}^{i-1}
(\theta^q-\theta)^{q^r}
=
-\sum_{r=0}^{i-1}L_1^{q^r}
\in\FF_p[L_1],
\]
for every $i\geq1$. We obtain that
\[
\exp_C(z)\in\FF_p(L_1)[\![z]\!],
\qquad
\frac{z}{\exp_C(z)}
\in\FF_p(L_1)[\![z]\!].
\]
It follows that
\[
\frac{\BC(j)}{\Gamma_{j+1}}
\in\FF_p(L_1),
\]
and therefore
\[
c_j
=
\frac{\BC(j)L_1^m}{\Gamma_{j+1}}
\in\FF_p(L_1)^\times.
\]
The final assertion follows immediately.
\end{proof}

For $\bullet\in\{\zeta,\Li\}$, Conditions
\ref{cond:unital} and \ref{cond:convolution} for the pair
\[
\left(
\mathscr{L}_\infty^\bullet,
\mathscr{L}_\infty^{\bullet,\ant}
\right)
\]
follow from \cref{lem:degree-convolution}, while Condition
\ref{cond:ev-product} follows from
\cref{thm:strict-bullet-character}. Hence
\cref{lem:image-comparison} implies
\[
\mathscr{L}_\infty^{\bullet,\ant}(\Hfk_R)
\subseteq
\mathscr{L}_\infty^\bullet(\Hfk_R)
=
\Zcal_{\infty,R},
\]
where the last equality for $\bullet=\Li$ follows from
\cref{thm:basis_for_infty}. We are now in a position to prove the following theorem.

\begin{thm}\label{thm:infty-coaction}
There exists a unique $R$-algebra homomorphism
\[
\widetilde{\Delta}_\infty:
\Zcal_{\infty,R}
\longrightarrow
\Zcal_{\infty,R}\otimes_R\ov{\Zcal}_{\infty,R}
\]
such that, for every
$\bullet\in\{\zeta,\Li\}$ and every
$\ww{s}\in\Ical$ with $\dep(\ww{s})=r$,
\begin{equation}\label{eq:infty-bullet-coaction}
\widetilde{\Delta}_\infty
\left(
\mathscr{L}_\infty^\bullet([\ww{s}])
\right)
=
\sum_{i=0}^r
\mathscr{L}_\infty^\bullet([\ww{s}[:i]])
\otimes
\pi_\infty
\left(
\mathscr{L}_\infty^\bullet([\ww{s}[i+1:]])
\right).
\end{equation}
Moreover, 
\begin{equation}\label{eq:infty-bullet-ant-coaction}
\widetilde{\Delta}_\infty
\left(
\mathscr{L}_\infty^{\bullet,\ant}([\ww{s}])
\right)
=
\sum_{i=0}^r
\mathscr{L}_\infty^{\bullet,\ant}([\ww{s}[i+1:]])
\otimes
\pi_\infty
\left(
\mathscr{L}_\infty^{\bullet,\ant}([\ww{s}[:i]])
\right).
\end{equation}
In particular, $\widetilde{\Delta}_\infty$ endows
$\Zcal_{\infty,R}$ with a right $\ov{\Zcal}_{\infty,R}$-comodule structure.
\end{thm}

\begin{proof}
Fix $\bullet\in\{\zeta,\Li\}$. We apply
\cref{prop:formal-coaction} with
\[
B=\Zcal_{\infty,R},
\qquad
J=\zeta_A(q-1)\Zcal_{\infty,R},
\]
and
\[
\ev=\mathscr{L}_\infty^\bullet,
\qquad
\ev^\ant=\mathscr{L}_\infty^{\bullet,\ant}.
\]
As observed above, Conditions \ref{cond:unital},
\ref{cond:convolution}, and \ref{cond:ev-product} hold. Moreover,
\[
\ker(\mathscr{L}_\infty^\bullet)
=
\mathscr{R}_R^\bullet
\]
by \cref{prop:Mishiba-kernel}. If $\bullet=\zeta$, then
\cref{lem:Euler-Carlitz-L1} implies that for any $j\in\NN$ with $q-1\mid j$, we have
\[
\pi_\infty
(
\mathscr{L}_\infty^\zeta([j])
)
=
0.
\]
If $\bullet=\Li$, then
\[
\mathscr{L}_\infty^\Li([q-1])
=
\Li_{q-1}(1)
=
\zeta_A(q-1).
\]
Hence, we have
\[
\pi_\infty
(
\mathscr{L}_\infty^\Li([q-1])
)
=
0.
\]
Thus \eqref{eq:quotient-vanishing} holds in both cases.

It follows from \cref{prop:formal-coaction} that, for each
$\bullet\in\{\zeta,\Li\}$, there exists a well-defined
$R$-algebra homomorphism
\[
\widetilde{\Delta}_\infty^\bullet:
\Zcal_{\infty,R}
\longrightarrow
\Zcal_{\infty,R}\otimes_R\ov{\Zcal}_{\infty,R}
\]
such that
\[
\widetilde{\Delta}_\infty^\bullet
\left(
\mathscr{L}_\infty^\bullet([\ww{s}])
\right)
=
\sum_{i=0}^r
\mathscr{L}_\infty^\bullet([\ww{s}[:i]])
\otimes
\pi_\infty
\left(
\mathscr{L}_\infty^\bullet([\ww{s}[i+1:]])
\right).
\]

We now show that
\[
\widetilde{\Delta}_\infty^\zeta
=
\widetilde{\Delta}_\infty^\Li.
\]
By Carlitz's formula (see \cite[Theorem~5.9.1]{Tha04}), for every $1\leq s\leq q$ and $d\geq 0$, we have 
\[
\mathscr{S}_d^\zeta(s)
=
\mathscr{S}_d^\Li(s).
\]
Hence the two coaction maps agree on
\[
\mathscr{L}_\infty^\zeta([\ww{s}])
=
\mathscr{L}_\infty^\Li([\ww{s}])
\]
for every Thakur index $\ww{s}\in\Ical^{\mathrm{T}}$. Since these elements form an $R$-basis of $\Zcal_{\infty,R}$
by \cref{thm:basis_for_infty}, the two coaction maps coincide on
$\Zcal_{\infty,R}$. We denote this common map by $\widetilde{\Delta}_\infty$ and the right
$\ov{\Zcal}_{\infty,R}$-comodule structure on $\Zcal_{\infty,R}$
follows immediately from \cite[Theorem~1.3.4]{Mis26}. Finally, applying \cref{prop:formal-coaction-ant} to both realizations
yields \eqref{eq:infty-bullet-ant-coaction} and the uniqueness part is clear.
\end{proof}
\begin{rmk}\label{rmk:F_p[L_1]doesnotwork}
The assumption that $R$ is an $\FF_p(L_1)$-algebra in
\cref{thm:infty-coaction} cannot in general be replaced by the assumption
that $R$ is an $\FF_p[L_1]$-algebra. For example, let $q=2$ and
$R=\FF_2[L_1]$. Expanding $\zeta_A(5)$ in the Thakur basis, we have
\[
\begin{aligned}
\zeta_A(5)
={}&
a\left(
\zeta_A(1,1,1,1,1)
+
\zeta_A(1,1,2,1)
\right)+
b\left(
\zeta_A(2,1,1,1)
+
\zeta_A(2,2,1)
\right),
\end{aligned}
\]
where
\[
a=L_1^8+L_1^5,
\qquad
b=L_1^8+L_1^6+L_1^5.
\]
If the conclusion of \cref{thm:infty-coaction} were valid over $R$, then,
since $(5)$ has depth one, the biweight $(2,3)$ component of
$\widetilde{\Delta}_\infty(\zeta_A(5))$ would be zero. On the other hand,
applying the coaction formula to the above expansion gives
\[
(a+L_1b)\,
\zeta_A(1,1)\otimes
\pi_\infty\left(
\zeta_A(1,1,1)+\zeta_A(2,1)
\right).
\]
Here, we use $\zeta_A(2)=L_1\zeta_A(1,1)$. Moreover, note that
\[
a+L_1b
=
L_1^5
\left(
L_1^4+L_1^3+L_1^2+L_1+1
\right),
\]
and
\[
\zeta_A(1)\zeta_A(1,1)
=
(L_1+1)
\left(
\zeta_A(1,1,1)+\zeta_A(2,1)
\right).
\]
Since $L_1+1\nmid a+L_1b$ in $\FF_2[L_1]$, the above biweight $(2,3)$
component is nonzero, a contradiction. Thus,
\cref{thm:infty-coaction} does not in general hold over an arbitrary
$\FF_p[L_1]$-subalgebra.
\end{rmk}

\subsection{The Coproduct and the Antipode}
\label{sec:infty-Hopf}

By \cite[Theorem~1.3.4]{Mis26}, Mishiba's coaction induces a coproduct
\[
\Delta_\infty:
\ov{\Zcal}_{\infty,R}
\longrightarrow
\ov{\Zcal}_{\infty,R}\otimes_R\ov{\Zcal}_{\infty,R}
\]
and an antipode
\[
S_\infty:
\ov{\Zcal}_{\infty,R}
\longrightarrow
\ov{\Zcal}_{\infty,R}.
\]
We now determine the actions of
$\Delta_\infty$ and $S_\infty$ on the $\infty$-adic multiple zeta
values.

\begin{thm}\label{thm:infty-Hopf-structure}
For every $\ww{s}=(s_1,\ldots,s_r)\in\Ical$, we have
\begin{equation}\label{eq:quotient-zeta-coproduct}
\Delta_\infty
\left(
\pi_\infty(\zeta_A(\ww{s}))
\right)
=
\sum_{i=0}^r
\pi_\infty(\zeta_A(\ww{s}[:i]))
\otimes
\pi_\infty(\zeta_A(\ww{s}[i+1:]))
\end{equation}
and
\begin{equation}\label{eq:antipode-zeta-to-ant}
S_\infty
\left(
\pi_\infty(\zeta_A(\ww{s}))
\right)
=
\pi_\infty
\left(
\zeta_A^\ant(\ww{s})
\right).
\end{equation}
In particular, \cref{conj:mishiba}~(ii) holds.
\end{thm}

\begin{proof}
By the construction of $\Delta_\infty$ in
\cite[Theorem~1.3.4]{Mis26}, we have
\[
\Delta_\infty\circ\pi_\infty
=
(\pi_\infty\otimes\id)\circ\widetilde{\Delta}_\infty.
\]
Hence \eqref{eq:quotient-zeta-coproduct} follows immediately from
\eqref{eq:infty-bullet-coaction} with $\bullet=\zeta$.

We prove \eqref{eq:antipode-zeta-to-ant} by induction on
$\dep(\ww{s})$. The case $\ww{s}=\varnothing$ is clear. Let
$\ww{s}=(s_1,\ldots,s_r)\in\Ical$ be non-empty. By the antipode identity and
\eqref{eq:quotient-zeta-coproduct},
\[
\begin{aligned}
S_\infty
\left(
\pi_\infty(\zeta_A(\ww{s}))
\right)
=
-\pi_\infty(\zeta_A(\ww{s}))-
\sum_{i=1}^{r-1}
\pi_\infty(\zeta_A(\ww{s}[:i]))
S_\infty
\left(
\pi_\infty(\zeta_A(\ww{s}[i+1:]))
\right).
\end{aligned}
\]
By the induction hypothesis, the right-hand side is
\[
-\pi_\infty(\zeta_A(\ww{s}))
-
\sum_{i=1}^{r-1}
\pi_\infty(\zeta_A(\ww{s}[:i]))
\pi_\infty
\left(
\zeta_A^\ant(\ww{s}[i+1:])
\right).
\]
On the other hand, \cref{lem:degree-convolution} gives
\[
\zeta_A^\ant(\ww{s})
+
\sum_{i=1}^{r-1}
\zeta_A(\ww{s}[:i])
\zeta_A^\ant(\ww{s}[i+1:])
+
\zeta_A(\ww{s})
=
0.
\]
Applying $\pi_\infty$ proves
\eqref{eq:antipode-zeta-to-ant}.
\end{proof}
\begin{rmk}\label{rmk:F_p[L_1]doesnotwork-antipode}
The assumption that $R$ is an $\FF_p(L_1)$-algebra in
\cref{thm:infty-Hopf-structure} cannot in general be replaced by the
assumption that $R$ is an $\FF_p[L_1]$-algebra. For example, let
$q=2$ and $R=\FF_2[L_1]$. By expanding the relevant values in the
Thakur basis, one obtains
\[
S_\infty
\left(
\pi_\infty(\zeta_A(1,5))
\right)
-
\pi_\infty
\left(
\zeta_A^\ant(1,5)
\right)
=
\pi_\infty(\zeta_A(6))
\neq 0.
\]
Thus, the antipode formula
\eqref{eq:antipode-zeta-to-ant} does not in general hold over an
arbitrary $\FF_p[L_1]$-subalgebra.
\end{rmk}

\section*{Acknowledgments}
The author would like to thank Yoshinori Mishiba for many helpful
discussions and valuable comments on the manuscript. He is also grateful to Chieh-Yu Chang for carefully
reading an earlier version of the manuscript and for valuable
suggestions that improved its presentation. This work was carried out
during the author's visit to Tohoku University in Sendai, which was
supported by the Japan-Taiwan Exchange Association. The author also
gratefully acknowledges the support of the National Science and
Technology Council over the past few years under grant no.\
113-2628-M-007-004.


\end{document}